\documentclass[12pt]{article}
\usepackage{amssymb,amsfonts,amsthm,amsmath}

\usepackage{MnSymbol}
\usepackage{pstricks, pst-plot,pst-node}
\usepackage{graphicx}
\usepackage{psfrag}

\usepackage[all]{xy}

\newcommand{\pic}[2]{\raisebox{-.5\height}{\includegraphics[scale=#2]{#1}}}

\def\Xor{\pic{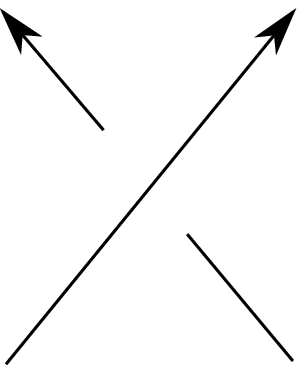} {.250}}
\def\Yor{\pic{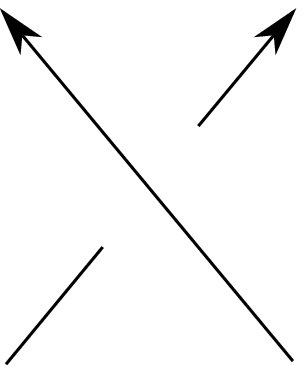} {.250}}
\def\Ior{\pic{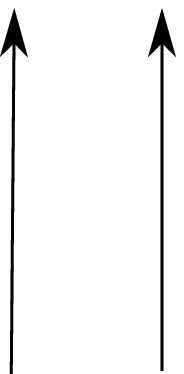} {.250}}

\def\unknot{\pic{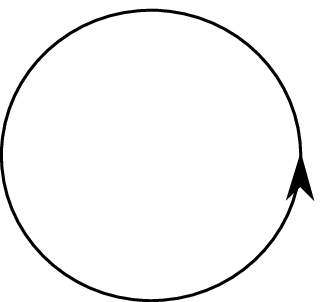} {.250}}

\def\Idor{\pic{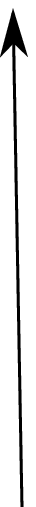} {.250}}
\def\Rcurlor{\pic{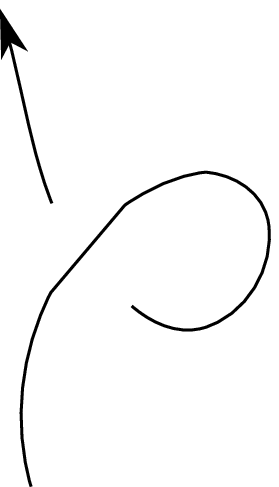} {.250}}
\def\Lcurlor{\pic{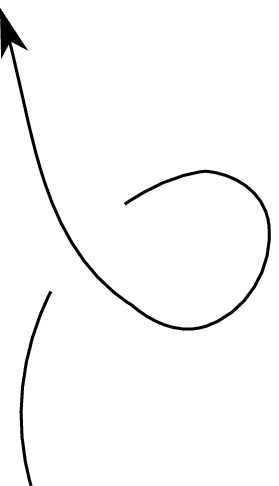} {.250}}

\newcommand{\bc}{\begin{center}}
\newcommand{\ec}{\end{center}}

\newcommand{\be}{\begin{equation}}
\newcommand{\ee}{\end{equation}}

\newcommand{\beqn}{\begin{eqnarray*}}
\newcommand{\eeqn}{\end{eqnarray*}}

\newcommand{\arrmu}{\overrightarrow{\mu}}
\newcommand{\arrlam}{\overrightarrow{\lambda}}
\newcommand{\DT}{\operatorname{DT}}
\newcommand{\PT}{\operatorname{PT}}

\newcommand{\sZ}{\mathsf{Z}}
\newcommand{\cL}{\mathcal{L}}
\newcommand{\calP}{\mathcal{P}}
\newcommand{\sP}{\mathsf{P}}
\newcommand{\sH}{\mathsf{H}}

\newcommand{\sM}{\mathsf{M}}
\newcommand{\sO}{\mathsf{O}}

\newcommand{\QQ}{\mathbb{Q}}
\newcommand{\CC}{\mathbb{C}}

\newcommand{\RR}{\mathbb{R}}

\newcommand{\sW}{\mathsf{W}}
\newcommand{\PP}{\mathbb{P}}

\newcommand{\ZZ}{\mathbb{Z}}

\newcommand{\NN}{\mathbb{N}}

\newcommand{\calO}{\mathcal{O}}
\newcommand{\calA}{\mathcal{A}}

\newcommand{\calC}{\mathcal{C}}

\newcommand{\calE}{\mathcal{E}}
\newcommand{\calF}{\mathcal{F}}

\newcommand{\calH}{\mathcal{H}}
\newcommand{\calI}{\mathcal{I}}

\newcommand{\calT}{\mathcal{T}}

\newcommand{\writhe}{\mathrm{wr}}
\newcommand{\sideperp}{\vdash}
\newcommand{\Tr}{\operatorname{Tr}}
\newcommand{\tr}{\mathrm{t}}
\newcommand{\tp}{\mathrm{top}}
\newcommand{\low}{\mathrm{low}}

\newcommand{\Hilb}{\operatorname{Hilb}}
\newcommand{\id}{\operatorname{id}}
\newcommand{\Ybar}{\overline{Y}}
\newcommand{\Cbar}{\overline{C}}

\theoremstyle{plain}
\newtheorem{lem}{Lemma}[section]
\newtheorem{thm}[lem]{Theorem}

\newtheorem{prop}[lem]{Proposition}

\newtheorem{cor}[lem]{Corollary}

\theoremstyle{definition}

\newtheorem*{remark}{Remark}

\begin{document}

\title{Stable pairs and the HOMFLY polynomial}
\author{Davesh Maulik}
\maketitle
\begin{abstract}
Given a planar curve singularity, we prove a conjecture of Oblomkov-Shende, relating the geometry of its Hilbert scheme of points to the HOMFLY polynomial
of the associated algebraic link.  More generally, we prove an
extension of this conjecture, due to Diaconescu-Hua-Soibelman, relating stable pair invariants on the conifold to the colored HOMFLY polynomial of the
algebraic link.  Our proof uses wall-crossing techniques to prove a blowup identity on
the algebro-geometric side.  We prove a matching identity for the colored HOMFLY polynomials of a link using skein-theoretic techniques.
\end{abstract}

\tableofcontents
\section{Introduction}
Given a planar curve singularity, Oblomkov and Shende \cite{oblomkov-shende} conjectured a precise relationship between
the geometry of its Hilbert scheme of points and the HOMFLY polynomial of the associated 
link.  In this paper, we give a proof of this conjecture, as well as a generalization to colored HOMFLY polynomial invariants, recently proposed by Diaconescu, Hua, and Soibelman
\cite{diaconescu-hua-soibelman}.
\subsection{Hilbert schemes of planar curve singularities}
We first recall the original formulation of \cite{oblomkov-shende}.
Let $C \subset \CC^2$ be a reduced curve, equipped with a distinguished point $p \in C$.
We can consider
the punctual Hilbert scheme $C_{p}^{[n]}$ of length $n$ subschemes of $C$ which are set-theoretically supported at $p$.
We define a constructible function
$$m: C_{p}^{[n]} \rightarrow \NN$$
such that, for a subscheme $Z \subset C$ supported at $p$, the value $m([Z])$ is the minimal number of 
generators of the defining ideal $I_{Z,p} \subset \widehat{\calO}_{C,p}$.
Consider the two-variable generating function
$$Z_{C,p}(v,s) = \sum_{n \geq 0} s^{2n} \int_{C_{p}^{[n]}}(1-v^2)^{m} d\chi.$$
In the above expression, $\int d\chi$ refers to the weighted Euler characteristic of the constructible function
on an algebraic variety, defined by
$$\int f d\chi = \sum_{n} n \chi_{\tp}(f^{-1}(n)).$$

On the other hand, let $\cL_{C,p}$ denote the link of the plane curve singularity at $p$, obtained by intersecting $C$ with a small
three-sphere around $p \in \CC^2$.  Given an oriented link $\cL \subset S^3$, the HOMFLY polynomial 
$\sP(\cL;v,s) \in \ZZ[v^{\pm 1}, (s - s^{-1})^{\pm}]$
is characterized
by the crossing relation
\bc
\qquad{$v \sP\left( \Yor\right)  -\ v^{-1}\sP\left(\Xor\right) \qquad =\qquad{(s-s^{-1})}\sP\left( \Ior\right) \ ,$}\ec
and the normalization
$$\sP(\mathrm{unknot}) = \frac{v - v^{-1}}{s - s^{-1}}.$$

Let $\mu$ denote the Milnor number of the singularity $(C,p)$.
Our first result is the following relation between these two generating functions, first stated as Conjecture $2'$ in \cite{oblomkov-shende}:
\begin{thm}\label{mainthm1}
$$\sP(\cL_{C,p};v,s) = \left(\frac{v}{s}\right)^{\mu - 1} Z_{C,p}(v,s).$$
\end{thm}

An immediate consequence of this theorem is that the right-hand side only depends on the topological type of the singularity, rather than its analytic structure.  
In Corollary \ref{severiobservation}, we will state an application (due to V. Shende) of this observation, for Severi strata of versal families of locally planar curves.

\subsection{Framed stable pairs}

In order to establish Theorem \ref{mainthm1}, we need a certain rephrasing of the conjecture, due to Diaconescu, Hua, and Soibelman \cite{diaconescu-hua-soibelman}.  Motivated by large $N$ duality considerations of the paper \cite{diaconescu-vivek-vafa}, 
the authors use a beautiful wall-crossing argument to reinterpret the above Hilbert scheme integrals in terms of the moduli space of stable pairs on the resolved conifold.

Let $$\pi: Y\rightarrow \PP^1$$ denote the total space of the rank two bundle
$$\calO_{\PP^{1}}(-1) \oplus \calO_{\PP^{1}}(-1)$$
on $\PP^1$.  It is a small resolution of a threefold double-point singularity, with exceptional locus given
by the zero section
$$E \cong \PP^1.$$

Let $\pi^{-1}(0)$ denote the fiber over $0 \in \PP^1$.
If we fix an identification of $(\CC^2,p)$ with $(\pi^{-1}(0), 0)$, we have an embedding
$C \hookrightarrow \pi^{-1}(0) \hookrightarrow Y$.

By definition, a stable pair on $Y$ is a
pure one-dimensional sheaf $\calF$ on $Y$ with a section
$$\sigma:  \calO_Y \rightarrow \calF,$$
such that
$\mathrm{Coker}(\sigma)$ is zero-dimensional.

We say that the stable pair is $C$-framed
if, upon restriction to the open subset $Y \backslash E$,
the pair $(\calF,\sigma)$ is isomorphic to the restriction of the surjection
$$\calO_{Y} \twoheadrightarrow \calO_{C}.$$

\begin{remark}
This definition differs slightly from that in \cite{diaconescu-hua-soibelman} since they allow nontrivial cokernel along $C$.  In their paper, $C$ has a unique singular point, so this does not make any major difference in formulas.
\end{remark}

Let $\Ybar$ be a projective compactification of $Y$ and $\Cbar$ the closure of $C$.  Any $C$-framed stable pair admits a unique extension to a $\Cbar$-framed stable pair $(\overline{\calF}, \overline{\sigma})$ on $\Ybar$.
Given integers $(r,n)$,
we define the moduli space $\calP(Y,C,r,n)$ of $C$-framed stable pairs $(\calF,\sigma)$
such that
\begin{enumerate}
\item[(i)] the support of $\calF$ has generic multiplicity $r$ along the zero section $E$
\item[(ii)] $\chi(\overline{\calF}) = n + \chi(\calO_{\Cbar})$.
\end{enumerate}
This moduli space is a locally closed subscheme of the space of stable pairs on $\Ybar$ and is independent of the choice of compactification.

We study the generating function
$$\sZ(Y,C;q,Q) = \sum_{r,n} q^{n}Q^{r} \chi_{\tp}(\calP(Y,C,r,n))$$
and its normalized version
$$\sZ'(Y,C;q,Q) = \frac{\sZ(Y,C;q,Q)}{\sZ_Y(q,Q)}$$
where
$$\sZ_Y(q,Q) = \sZ(Y,\emptyset;q,Q) = \prod_{k} (1+q^{k}Q)^{k}$$
is the topological PT series of the resolved conifold.

Then Theorem $1.1$ of \cite{diaconescu-hua-soibelman} is the equality
$$\sZ'(Y,C;q,Q)= Z_{C,p}(v,s)$$
after the change of variables $q = s^{2}, Q = -v^{2}$. 
\subsection{Colored variant}

The advantage of working with stable pairs is that it allows us to formulate a more general statement, first conjectured in \cite{diaconescu-hua-soibelman}, following work of Oblomkov-Shende in the $Q=0$ limit.

Suppose $C$ consists of irreducible components $C_1, \dots, C_r$ and that, for each component $C_i$
we have associated a partition 
$$\mu_i = \{\mu^{(1)}_{i}, \dots, \mu^{(\ell_{i})}_{i}\}.$$
We will use the shorthand
$$\arrmu$$ to denote the vector of partitions $(\mu_1,\dots, \mu_r)$.
For simplicity of notation, we will always 
assume\footnote{Up to formal isomorphism, every planar curve singularity can be expressed in this form, so this is a harmless assumption.}
 in this paper that each irreducible component of $C$ gives a unique analytic branch of $C$ at $p$.
Fix a toric affine chart of $0 \in Y$ with coordinates $(x,y,z)$ such that the projection $\pi$ is given by projection to the $z$-axis.
Let $f_i(x,y)$ be the defining equation of $C_i$; we define 
the pure one-dimensional subscheme 
$C_{i, \mu_{i}}$ by the 
equations
$$z^{j-1}f_{i}(x,y)^{\mu^{(j)}_{i}}\, j = 1, \dots, \ell_{i}.$$
Let $C_{\arrmu}$ denote the unique pure one-dimensional subscheme of $Y$
supported on $C$ and which agrees with $C_{i,\mu_{i}}$ generically on each $C_i$.

As before, we define $C_{\arrmu}$-framed stable pairs on $Y$ by the condition that, on the complement of $E$,
they agree with the surjection $\calO_Y \twoheadrightarrow \calO_{C_{\arrmu}}$.
Similarly, we define moduli spaces
$\calP(Y, C, \arrmu; r,n)$ and the associated generating functions
$$\sZ(Y, C, \arrmu; q, Q),\quad \sZ'(Y, C, \arrmu; q, Q).$$
Notice that these only depend on the formal neighborhood of $C$ at $p$.

On the knot side, given an oriented link $\cL$, if one labels each component $\cL_i$ with a partition $\lambda_i$,
there exists a colored variant of the HOMFLY polynomial
$$\sW(\cL, \arrlam; v,s).$$  
We will give a definition of this invariant via skein theory in Section \ref{coloredhomfly}.
After the specialization $v = s^{-N}$, these invariants correspond to Reshetikhin-Turaev knot invariants associated to the quantum group $U_q(\mathfrak{sl}(N))$.  In this picture, the partition labels indicate the representation associated to each strand.
When all partitions are $(1)$, we recover the original HOMFLY polynomial after a change of normalization; see Section \ref{coloredhomfly} for the precise relation.

The relation between stable pairs invariants and colored HOMFLY invariants of algebraic links $\cL = \cL_{C,p}$ is as follows.
Given $C$ and $\arrmu$ as above, components of $\cL$ are in bijection with branches of $C$ at $p$; we label each component $\cL_i$ with the transpose of the partition $\mu_{i}$ associated to the corresponding branch.  
The main theorem of this paper is then the following
\begin{thm}\label{mainthm2}
There exist integers $a(C, \arrmu)$, $b(C,\arrmu)$ and a sign $(-1)^{\epsilon(C,\arrmu)}$
such that
$$\sZ'(Y,C,\arrmu;q,Q) =  (-1)^{\epsilon(C,\arrmu)} v^{a(C,\arrmu)} s^{b(C,\arrmu)} \sW(\cL,\arrmu^{\tr};v,s)$$
after the change of variables
$$q= s^{2},\quad Q = -v^{2}.$$
\end{thm}

That is, the two quantities differ by a uniquely determined monomial shift.  The arguments in this paper can be used to express this shift in terms of the data of $C$ and $\arrmu$, but we will only write it out explicitly in the uncolored case, which gives a proof of Theorem \ref{mainthm1}.

\subsection{Strategy and outline}
The idea of the proof is as follows.
Using embedded resolution of singularities, after a sequence of blowups of $\CC^2$, the total transform of the singularity $(C,p)$ has at worst nodal singularities.
In these cases, the $Q = 0$ specialization of Theorem \ref{mainthm2} is a well-known consequence of the 
topological vertex formalism; see, for instance, 
\cite{akmv, orv} for an overview of these ideas.  In order to reduce to this case, we prove a blowup formula for the framed stable pairs generating function
associated to $C$ and an analogous formula for the colored HOMFLY polynomial of its link.

In order to construct a blowup identity for $\sZ'(Y,C, \arrmu;q,Q)$, the main tool is to prove the invariance of the generating function with 
respect to a simple flop; these flop invariance identities are standard and, in the context of Gromov-Witten theory, go back to \cite{li-ruan}.  In our case, we use wall-crossing techniques, as in the papers \cite{toda, calabrese}.  Even in the uncolored case, the blowup identity requires arbitrary partition labels on the total transform.

For the corresponding identity on the link side, the idea is to use linearity of the HOMFLY trace to reduce to the
case of the colored unknot, where it can be shown directly (or by moving back to the algebraic geometry side).  Finally, we need to prove the compatibility 
of the two identities with Theorem \ref{mainthm2}; the main issue here is keeping track of the various monomial shifts in the statement.  Once this
is established, an inductive argument reduces to the nodal case, and finishes the proof.

We give a brief outline of the paper.  In section $2$, we prove the flop invariance statement and translate it into a blowup formula for $\sZ'(\dots)$.
In sections $3$ and $4$, we recall some basic features of the skein-theoretic approach to the HOMFLY polynomial (following the work
of Morton and collaborators \cite{aiston-morton, lukac-morton, morton-manchon})
and to the links associated plane curve singularities (following the book of Eisenbud-Neumann \cite{eisenbud-neumann}).  
In section $5$, we use this to prove a blowup identity for $\sW(\dots)$.  In section $6$, we combine the two identities and check the
needed compatibilities.
In the appendix, we prove a certain Hecke character identity needed for the skein calculations in section $3$.

\subsection{Further directions}

In \cite{oblomkov-rasmussen-shende}, the authors conjecture a refined version of Theorem \ref{mainthm1}, relating the weight filtration on cohomology 
to the HOMFLY homology of $\cL_C$, as defined by Khovanov-Rozansky \cite{khovanov-rozansky}.  It is natural to ask
whether the approach here can be refined as well.

On the link side, we are not aware of any version of skein model techniques for HOMFLY homology.  
However, for torus knots, there are conjectural descriptions in terms of an operator formalism on Fock space, due to Aganagic-Shakirov \cite{mina-shamil} and Cherednik \cite{cherednik}, which may generalize to arbitrary algebraic links.  If provable, this formalism would effectively replace the techniques used here.

On the algebro-geometric side, the relation with stable pairs is conjecturally extended to motivic invariants in \cite{diaconescu-hua-soibelman}.  In addition to standard foundational
issues regarding the motivic wall-crossing machine (e.g. existence and compatibility of orientation data), they also 
require additional conjectures regarding the behavior of the motivic weight.
Under the assumption of the same conjectures, it seems plausible that the blowup identity may extend as well.  One would also 
need to understand the colored Hopf link, which seems closely related to the conjectural matching of motivic and refined invariants for toric Calabi-Yau geometries, as proposed in \cite{balazs-paper} and \cite{gukov-dimofte}.

\subsection{Acknowledgements}

We thank Brian Conrad, Sabin Cautis, Eugene Gorsky, Adam Knapp, Max Lieblich, Robert Lipshitz, Walter Neumann, Andrei Okounkov, Rahul Pandharipande, Sucharit Sarkar, Noah Snyder, and Ben Young for many enlightening discussions.  We are especially grateful to John Calabrese, Emanuel Diaconescu, Hugh Morton, and Alexei Oblomkov for a detailed discussion of their papers, and to Vivek Shende for comments on an earlier draft. The author has been partially supported by a Clay Research Fellowship and by NSF Grant DMS-1159416.

\section{Blowup identity}

In this section, we prove a blowup formula for 
$\sZ'(Y,C,\arrmu;q,Q)$ via flop invariance of stable pairs generating functions.

\subsection{Setup}\label{blowupsetup}
We will follow all notation from the introduction.

Let $Y_{-}$ be the local Calabi-Yau threefold obtained via simple flop of the $(-1,-1)$ curve $E$; 
it is again given by the total space of $\calO(-1) \oplus \calO(-1)$ on a rational curve $E_{-}$ and there
exists a birational morphism $\phi: Y \dashedrightarrow Y_{-}$ defined away from $E$ and $E_{-}$

The proper transform of $\pi^{-1}(0)$ with respect to $\phi$ is given by its blowup at the origin, with exceptional fiber $E_{-}$.  Similarly,
the proper transform of $C$ with respect to $\phi$ is the blowup $C'$ of $C$ at $p$.
It intersects $E_{-}$ at points $p_{1}, \dots, p_{e}$ corresponding to the tangent cone of $C$ at $p$.  For
$k = 1, \dots, e$, let 
$$(D_{k},p_{k})$$ 
denote the singularity of the (reducible) curve $C' \cup E_{-}$ at $p_{k}$.  These are planar singularities, and the output of this section
will be a formula for $\sZ(Y,C,\arrmu;q,Q)$ in terms of the corresponding generating functions for $D_k$.

Given a sequence of partitions $\arrmu$ labelling the components of $C$, let $C'_{\arrmu} \subset Y_{-}$ denote the pure scheme-theoretic closure of $\phi(C_{\arrmu}) \subset
Y_{-}\backslash E_{-}$.

The definition of $C'_{\arrmu}$-framed stable pairs is analogous to the case of $Y$ and we have
generating functions
$$\sZ(Y_{-}, C', \arrmu; q, Q),\quad \sZ'(Y_{-}, C', \arrmu;q, Q).$$

In what follows, let $m_{1}, \dots, m_{r}$ be the multiplicities of the branches of $C$ at $p$ and
let $\mu = \sum_{i=1}^{r} m_i \mu_i$
be the partition whose first part is $\mu^{(1)} = \sum m_i \mu_{i}^{(1)}$, etc.
We will use wall-crossing techniques to prove the following flop identity for the normalized generating functions.

\begin{prop}\label{toricflopformula}
We have the flop identity
\begin{equation}\label{flop1}
Q^{|\mu|} \sZ'(Y, C, \arrmu; q, Q^{-1}) = q^{\delta} \sZ'(Y_{-}, C', \arrmu; q, Q)
\end{equation}
where
$$
\delta = \sum_{j=1}^{l(\mu)} \binom{\mu^{(j)}}{2} - (j-1)\,.
$$
\end{prop}
The precise value of $\delta$ will not matter for our arguments.

\subsection{Compact geometry}

In order to apply the wall-crossing results without modification, it is convenient to work
with compact Calabi-Yau geometries.
While there is no Calabi-Yau compactification of $Y$, there do exist compact Calabi-Yau threefolds $X$
such that the formal (or analytic, even) neighborhood of $\pi^{-1}(0) \cup E$ embeds into $X$.  

Following \cite{diaconescu-hua-soibelman}, we fix one such geometry as follows.
Let $S$ be the blowup of $\PP^2$ at a point, and let
$$\rho: X_{-} \rightarrow S$$ be a smooth Calabi-Yau threefold that is a Weierstrass elliptic fibration over $S$.  
See, for instance, \cite{vafa-morrison} or \cite{fullwood} for constructions of such families.
Let $S_{-}$ denote the identity section of $\rho$, and let $E_{-} \subset S_{-}$ be the unique $(-1)$-curve, which gives a 
$(-1,-1)$-curve in the total space $X_{-}$.

\begin{lem}
The formal completion $\widehat{S_{-}}$ of $X_-$ along $S_{-}$ is isomorphic to the completion of the normal bundle of $S_-$ along the zero section.
\end{lem}
\begin{proof}
This can be proven via direct calculation with the Weierstrass equation, but  one quick argument is to use that $\widehat{S_{-}}$ is a one-dimensional commutative formal group scheme over $S_{-}$ so the logarithm gives the desired isomorphism with the formal completion inside its Lie algebra, i.e. the normal bundle with the additive group structure.  
\end{proof}

Let $X_+$ with exceptional curve $E_{+}$ be the simple flop of $X_-$ with respect to the curve $E_-$, and let $S_+$ denote the proper transform of $S_{-}$.  It intersects $E_+$ transversely at a single point $p$.  Using the above formal isomorphism, we see that the formal completion along $S_{+}\cup E_{+}$ contains the formal completion of $Y_{+}$ along  $\pi^{-1}(0) \cup E$ as an open formal subscheme.

Given the subscheme $C_{\arrmu} \subset Y$, we use this formal identification and take the scheme-theoretic closure to obtain a subscheme $\overline{C_{\arrmu}} \subset X_+$.  We can do the same thing to define $\overline{C'_{\arrmu}} \subset X_{-}$.
As in the introduction, we define $\overline{C_{\arrmu}}$-framed stable pairs on $X$
and their moduli space $\calP(X_{+},\overline{C}, \arrmu,r,n)$ where 
we now choose the convention $n = \chi(\calF)$ for discrete invariants.
If we define
$$\sZ(X_{+}, \overline{C}, \arrmu; q, Q) = \sum_{r,n} Q^{r}q^{n} \chi_{\tp}(\calP(X_+, \overline{C}, \arrmu, r, n)),$$
and the corresponding normalized generating function $\sZ'(\dots)$,
then
the formal isomorphism (and our change of discrete invariant convention) gives the equality
$$\sZ'(X_{+},\overline{C}, \arrmu; q, Q) = q^{\chi(\calO_{\overline{C}_{\arrmu}})} \sZ'(Y, C, \arrmu; q, Q)$$
and
$$\sZ'(X_{-}, \overline{C'}, \arrmu; q, Q) = q^{\chi(\calO_{\overline{C'}_{\arrmu}})} \sZ'(Y_{-}, C', \arrmu; q, Q).$$

An Euler characteristic calculation shows that
$$\chi(\calO_{\overline{C'_{\arrmu}}}) = \chi((\calO_{\overline{C_{\arrmu}}}) + \delta\,.$$
Therefore, in order to prove Proposition \ref{toricflopformula},
it suffices to show
\begin{prop}\label{compactflop}
$$Q^{|\mu|}\sZ'(X_{+},\overline{C}, \arrmu; q, Q^{-1}) = \sZ'(X_{-}, \overline{C'}, \arrmu; q, Q)\,.$$
\end{prop}

\subsection{Flop invariance via wall-crossing}

If we work with the moduli space of stable pairs, without any framing condition, flop invariance of the associated generating
function is shown via wall-crossing techniques in \cite{toda} and \cite{calabrese}; see \cite{hu-li} for a global argument for virtual invariants.  The
wall-crossing approach has the advantage that it applies equally well
in the virtual or topological setting and behaves well with respect to stratification.  As a result, the same arguments apply here
with only minor modification.  We recall the argument of \cite{calabrese} and explain the changes needed to include the framing condition.

The key idea is to study perverse coherent sheaves in the sense of \cite{bridgelandflop}.  
For $p= -1,0$, these are defined by certain perverse $t$-structures on $X_{\pm}$ with hearts $\calA^{p}_{\pm}$,
obtained by tilting with respect to suitable torsion pairs $(\calT^p_{\pm}, \calF^p_{\pm})$.
It follows from definitions that all objects in $\calF^p_{\pm}$ are supported on the exceptional curves $E_{\pm}$ and, in particular,
 objects in $\calA^p_{\pm}$ are simply sheaves when restricted to the complement.  There exists
a derived equivalence $\Phi: D^b(X_+) \rightarrow D^b(X_-)$ that induces an equivalence $\calA^{-1}_{+} \rightarrow \calA^{0}_{-}$.
We refer the reader to \cite{calabrese} and \cite{bridgelandflop} for more details and definitions.

In particular, given $\beta \in H_{2}(X_{\pm})$ and $n \in \ZZ$, we can study the perverse Hilbert scheme
$\Hilb^p(X)_{\beta,n}$, parametrizing quotients of $\calO_X$ in
$\calA^p_{\pm}$ with $\mathrm{ch}_2 = \beta$, Euler characteristic $n$, and at-most one-dimensional support.
We have the associated generating function
$$\DT^p(X_{\pm}; Q,q) = \sum_{\beta, n} \chi_{\tp}(\Hilb^p(X)_{\beta,n}) Q^{\beta}q^{n} \in \Lambda^p_{\pm}.$$
Here $\Lambda^p_{\pm}$ is a certain completion of the ring 
$$\left\{\sum c_{\beta,n} Q^{\beta} q^{n} | (\beta,n) = (\mathrm{ch}_2(E),\chi(E)), E \in \calA^p_{\leq 1,\pm}\right\}.$$
Let
$\phi_*: H_{2}(X_+) \rightarrow H_{2}(X_-)$ be the 
isomorphism induced by the flop $\phi$; it induces a map $\Lambda^{-1}_{+} \rightarrow \Lambda^0_{-}$.
We have an equality
\begin{equation}\label{perverseflop}
\DT^{p=0}(X_{-}; Q, q) = \phi_{*}\left(\DT^{p=-1}(X_+;Q,q)\right)\,.
\end{equation}

The desired flop equality comes from relating each side of \eqref{perverseflop} to the Hilbert scheme for the standard
$t$-structure.  Let $\sH(\calA^p_{\leq 1})_{\Lambda}$ denote the motivic Hall algebra  of the moduli stack of 
one-dimensional perverse coherent sheaves (or, rather, an appropriate completion of it).  Calabrese proves an identity
\begin{equation}\label{hallidentity}
\Hilb^p \ast 1_{\calF^{p}[1]} = 1^{\calO}_{\calF^{p}[1]}\ast \Hilb\,,
\end{equation}
Here, $\Hilb^p$ and $\Hilb$ are elements of the Hall algebra arising from the corresponding Hilbert schemes,
and
$1_{\calF^{p}[1]}$ and $1^{\calO}_{\calF^{p}[1]}$
are the Hall algebra elements associated to the stack of complexes in $\calF^{p}[1]$, equipped with a section 
in the latter case.  One can process this identity further by writing $1^{\calO}_{\calF^{p}[1]}$
in terms of $1_{\calF^{p}[1]}$ and stable pairs supported on $E$.

There is an integration map
$$\calI:  \sH_{\mathrm{reg}}(\calA^p_{\leq 1,\pm})_{\Lambda} \rightarrow \Lambda^p_{\pm}$$
defined on the subalgebra generated by schemes over the moduli stack by taking their 
topological Euler characteristic (as opposed to any Behrend weight).  
Although $1_{\calF^{p}[1]}$ is not in this subalgebra,  conjugation by
it preserves this subalgebra and one can show that this conjugation operator becomes trivial after applying $\calI$. 
The upshot of this analysis is
that 
$$\DT^{p}(X_{\pm};Q,q) = \sZ_Y(q,Q^{-[E_{\pm}]})\cdot \DT(X_{\pm}; Q,q).$$
Recall that
$$\sZ_Y(q,Q) = \prod_{k} (1+q^{k}Q)^{k}$$
is the topological PT series of the resolved conifold
and $\DT(X;Q,q)$ is the topological DT series of $X$, i.e. the generating function of topological Euler characteristics
of the usual Hilbert scheme of curves. 
Notice that this formula is independent of $p = -1,0$ and is stated for both $X_+$ and $X_-$.
 In combination with
equation \eqref{perverseflop}, this yields
$$\frac{1}{\sZ_Y(q, Q^{[E_{-}]})} \DT(X_{-};Q,q) = 
\phi_{*}\left(\frac{1}{\sZ_Y(q, Q^{[E_{+}]})}\DT(X_{+};Q,q)\right).$$

\subsection{Proof of Proposition \ref{compactflop}}

It suffices to modify the above argument to handle framed stable pairs.  
Given an element of a perverse or standard Hilbert scheme, we say it is $C_{\arrmu}$-framed if its restriction to $X\backslash E$ is the restriction of the subscheme $\calO_{C_{\arrmu}}$.
As before, we consider the subspace of framed objects and the corresponding generating functions of topological Euler characteristics.
First, since the flopped identification of perverse Hilbert schemes is compatible with the framing, we immediately have a framed analog of equation \eqref{perverseflop}:
$$\DT^{p=0}(X_{-}, C', \arrmu; q, Q) = \phi_{*}\left(\DT^{p=-1}(X_+, C, \arrmu;q,Q)\right)\,. $$

Second, we require the framed version of the Hall identity \eqref{hallidentity}.
The original statement is proven through a sequence of intermediate stacks, each step is either a geometrically bijective morphism or comparison of Zariski-locally-trivial fibrations with isomorphic fibers.  For example, the first step is as follows.

Let
$\mathcal{M}_L$ denote the stack of diagrams in $\calA^p$
\begin{equation*}
\xymatrix{
& \mathcal{O}_X \ar[d]_{\phi} &  & &
\\
0 \ar[r] & P_1 \ar[r] & P \ar[r] & P_2 \ar[r] & 0
\\
}
\end{equation*}
with the vertical map surjective in $\calA^p$ and $P_2 \in \calF^p[1]$
and
$\mathcal{M}'$ denote the stack of diagrams
$$\psi:\calO_X \rightarrow P$$
with $\mathrm{Coker}(\psi) \in \calF^{p}[1]$.
There is a geometrically bijective morphism that sends the first diagram to $\calO_X \rightarrow P$.
We replace both stacks by the closed substacks where
$\phi$ and $\psi$ are isomorphic to the surjection 
$$\calO \twoheadrightarrow \calO_{C_{\arrmu}}$$
when restricted to the open set $X \backslash E$.
It is clear that the above bijective morphism restricts to a geometrically bijective morphism between these closed substacks.
A similar check works for the other steps.

The result is the Hall identity for framed Hilbert schemes:
$$\Hilb^p_{\mathrm{framed}} \ast 1_{\calF^{p}[1]} = 1^{\calO}_{\calF^{p}[1]}\ast \Hilb_{\mathrm{framed}}\,,$$
Once this identity is in place, the remaining analysis is identical to \cite{calabrese} and yields the relation
\begin{equation}\label{dtframedflop}
\frac{1}{\sZ_Y(q, Q^{[E_{-}]})} \DT(X_{-}, C', \arrmu;q,Q) = 
\phi_{*}\left(\frac{1}{\sZ_Y(q, Q^{[E_{+}]})}\DT(X_{+}, C, \arrmu;q,Q)\right)\,.
\end{equation}

Finally, we need to pass from framed subschemes of $X_{\pm}$ to framed stable pairs on $X_{\pm}$.  Here, we have the relation
\begin{lem}
$$\DT(X_{+}, C, \arrmu;q,Q) = \sZ(X_{+}, C, \arrmu; Q, q) \cdot M(q)^{\chi_{\tp}(E_{+})}$$
where
$$M(q) = \prod_{k \geq 1} \frac{1}{(1-q^{k})^{k}}$$
is the Macmahon function.
\end{lem}
The analogous relation holds for $X_{-}$.
\begin{proof}
This is basically proven in \cite{stoppa-thomas}, which gives a punctual statement at the end of section $4.10$.
For any threefold $X$, fix a Cohen-Macaulay curve $B \subset X$ and a point $p \in X$, not necessarily lying on $B$,
and consider subschemes (or stable pairs) with underlying curve $B$ for which
the maximal zero-dimensional subsheaf  (or the cokernel in the stable pairs case) is set-theoretically supported at $p$. In other words, these are 
subschemes (and stable pairs) that are $B$-framed on $X -  \{p\}$.
The authors show the wall-crossing relation
$$\DT(X, B,p)(q) = \PT(X, B, p)(q)\cdot M(q),$$
where $\DT(X,B,p)(q)$ is the generating function of Euler characteristics of the stratum of subschemes of $X$, $B$-framed on $X-\{p\}$ (and similarly
for $\PT$).

If we fix $B$ but allow zero-dimensional behavior along $E$, then the proof of \cite{stoppa-thomas} applies with only the change that we take the category $\calT(E)$ of 
zero-dimensional sheaves set-theoretically supported on $E$, instead of the punctual category $\calT(p)$.  The result is the formula
 $$\DT(X,B, E) = \PT(X,B,E) \cdot F_E(q).$$ 
Here, the factor $F_E(q)$ depends on $E$ but is independent of $B$.  If we set $B = \emptyset$, $\DT(X,B,E)$ encodes Euler characteristics of the Hilbert scheme of points on $X$ supported on $E$, so we have the evaluation
$$F_E(q) = M(q)^{\chi_{\tp}(E)}.$$

Finally, section $4.2$ of \cite{stoppa-thomas} explains how to integrate this identity over any constructible locus of Cohen-Macaulay curves $B \subset X$; in our case we take Cohen-Macaulay curves which agree with $C_{\arrmu}$ on the complement of $E$, which gives the result.
\end{proof}

Since $\chi_\tp(E_{+}) = \chi_\tp(E_{-})$, this lemma implies the PT version of equation \eqref{dtframedflop}.

In order to complete the proof of Proposition \ref{compactflop}, we need to understand the action of $\phi_{*}$ on the curve classes
$[C_{\arrmu}]$ and $[E_{+}]$.  We first have that  $\phi_{*}([E_{+}]) = - [E_{-}]$.
We also have that
$$\phi_{*}([C_{\arrmu}]) = [C'_{\arrmu}] + |\mu|[E_{-}] \in H_{2}(X_{-})\,.$$
This can be determined via calculating the local intersection multiplicity of both sides with a divisor meeting $E_{-}$ transversely away from $C' \cap E$.
This concludes the proof.

\subsection{Localized flop identity}

Consider the action of $T = \CC^*$ on $Y_{-}$ that fixes the proper transform $S_{-}$ of $\pi^{-1}(0) \subset Y$ and scales its normal bundle.
By construction, it is compatible with the framing, i.e. $C'_{\arrmu}$ is a $T$-fixed subscheme.

We can rewrite the right-hand side of Proposition \ref{toricflopformula} by taking the $T$-fixed points of $\calP(Y_{-},C', \arrmu; r,n)$.
Given a $T$-fixed, $C'_{\arrmu}$-framed stable pair 
$$\calO_{Y_{-}} \rightarrow F$$
on $Y_{-}$, 
the support of $F$ is a $T$-fixed Cohen-Macaulay subscheme.  In particular, the non-reduced structure at a generic point of $E_{-}$ is determined
by a partition $\lambda \sideperp r$.  We choose our partition conventions to match those of $C'_{\arrmu}$, i.e. $\lambda^{(1)}$ determines the nonreduced structure along $S_{-}$, etc.
Given $\lambda$, let $E_{\lambda}$ denote the corresponding thickening of $E_{-}$.
Let 
$$D_{\lambda} \subset Y_{-}$$
be the unique $T$-fixed Cohen-Macaulay subscheme
that agrees with $C'_{\arrmu}$ and $E_{\lambda}$ at their generic points.  

The closed subset of $\calP(Y_{-}, C', \arrmu; r, n)$ with underlying subscheme $D_{\lambda}$ can be stratified based on the zero-dimensional support of the cokernel.  As a result, the topological Euler characteristic of this subset can be determined in terms of punctual contributions based on the types of singularities of $D_{\lambda}$.  
As we vary $\chi(\calE)$, the generating function factors into the corresponding punctual generating functions, which we determine as follows.

Recall by assumption that each irreducible component of $C$ has a unique branch at $p$.  Therefore, for each such component $C_{i}$ (for $i = 1, \dots r$), the proper transform $C'_{i}$ meets $E_{-}$ at a unique point in $\{p_{1}, \dots, p_{e}\}$ with multiplicity $m'_{i}$.
In particular, we have a decomposition of the $r$-tuple of partitions $\arrmu = (\mu_1, \dots, \mu_r)$ into $e$ disjoint subsets
$$\arrmu = \bigsqcup_{k=1}^{e} \arrmu[k]$$
corresponding to which branches which meet $p_k$.

For $1 \leq k \leq e$, the punctual contribution at $p_k$ is determined by the reduced plane curve singularity $D_k$ at $p_{k}$, decorated by the tuple of partitions $(\arrmu[k], \lambda)$.  
This is precisely the $Q=0$ contribution of the framed conifold generating functions defined in the introduction, i.e.
the punctual contribution is
$$\sZ(Y, D_k, (\arrmu[k],\lambda); q, Q=0).$$
That is, we choose an algebraic plane curve with singularity $D_k$ at the origin, and proceed as before.
We set $Q=0$ because we only consider stable pairs supported on the curve singularity itself, rather than allowing any exceptional curve support.  Since $\sZ_Y(q, Q=0) = 1$, we can replace this with the normalized series $\sZ'$.

The other term comes from the punctual contribution of the stratum $E_{-} \backslash \{p_1, \dots, p_e\}$.  
This is given by 
$$H_{\lambda}(q)^{\chi_{\tp}(E_{-} \backslash \{p_{1}, \dots, p_{e}\})} = H_{\lambda}(q)^{2-e}$$
where $H_{\lambda}(q)$ is the contribution of a single point $p \in E_{-} \backslash \{p_1, \dots, p_e\}$.

We have a formula for $H_{\lambda}(q)$, namely the one-leg Calabi-Yau stable pairs vertex, studied in \cite{pt-vertex}.  After a monomial shift so its constant term is $1$, the formula is given by the one-leg topological vertex
$$s_{\lambda}(q^{\rho}) = (-1)^{|\lambda|} q^{\frac{1}{2}h(\lambda) + \frac{1}{4}\kappa_{\lambda}} H_{\lambda}(q)\,,$$
where $q^{\rho}$ denotes the substitution $z_i = q^{-i + 1/2}$ for $i = 1, 2, \dots.$  Since we do not need this equality now, we will recall its proof later in Lemma \ref{basecase}.

The only remaining term to calculate is the monomial normalization determined by the monomial of lowest degree.
The punctual contributions we have written down have constant term $1$, which is the contribution corresponding to trivial cokernel.

In the right-hand side of Proposition \ref{toricflopformula}, the stable pair 
$$\calO_{Y_{-}} \twoheadrightarrow \calO_{D_{\lambda}}$$
contributes to the coefficient of the monomial
$$Q^{|\lambda|}q^{f(\lambda,\arrmu)}$$
where $f(\lambda,\arrmu)$ is determined as follows.

Fix a compactification of $Y_{-}$ and let $\overline{D}_{\lambda}$ and $\overline{C'}_{\arrmu}$ denote the scheme-theoretic closures inside $\overline{Y}_{-}$.
By definition, we have
$$f(\lambda,\arrmu) = 
\chi(\calO_{\overline{D}_{\lambda}}) - \chi(\calO_{\overline{C'}_{\arrmu}}).$$
It is independent of the choice of compactification.

\begin{prop}\label{realblowupidentity}
\begin{align*}
Q^{|\mu|}\sZ'(Y, C, \arrmu; q,Q^{-1}) = 
\frac{1}{\sZ_Y(q,Q)} \sum_{\lambda}
&Q^{|\lambda|}\cdot q^{f(\lambda,\arrmu) + \delta} \cdot\\ &H_{\lambda}(q)^{2-e}\prod_{k=1}^{e} \sZ'(Y, D_{k}, (\arrmu[k], \lambda); q, Q=0) .
\end{align*}

\end{prop}

Regarding the monomial shift $f(\lambda,\arrmu)$,
we can give a precise formula for it as follows.

Given a partition $\mu_{i}$,
let
$$(\mu_{i},\lambda) = \sum_{j} \mu_{i}^{(j)} \cdot \lambda^{(j)}$$
be the dot product with $\lambda$.

\begin{lem}
We have
\begin{align*}
f(\lambda,\arrmu) = \chi(\calO_{E_{\lambda}}) - \sum_{i=1}^{r} m'_{i} (\mu_{i}, \lambda) = h(\lambda) - \sum_{i=1}^{r} m'_{i} (\mu_{i},\lambda)
\end{align*}
where $$h(\lambda) = \sum_{\Box} h(\Box)$$
is the sum of the hook-lengths of $\lambda$.
\end{lem}

\section{Skein models}

In this section, we recall definitions and basic facts about the framed HOMFLY skein of a surface, following papers of Morton and collaborators \cite{aiston-morton, lukac-morton, morton-manchon}.
In particular, we give a definition of the colored HOMFLY polynomial in this formalism and summarize the previous calculations from their papers that we
will need here.  
For our purposes, we only need to understand the lowest degree terms of these calculations but we state the full answers anyways for the sake of completeness.
Nearly all of the material in this section is known to experts; we refer the reader to the references, especially \cite{lukac-thesis}, for more details.

\subsection{Recap of skein theory}\label{skeinrecap}

Fix the coefficient ring
$$\Lambda = \ZZ\left[v^{\pm 1}, s^{\pm 1}, \frac{1}{s^{r}- s^{-r}}, r\geq 1\right].$$ 

For a planar surface $F$ with boundary and designated input and output boundary points, the framed Homfly skein over $\Lambda$
is defined as the $\Lambda$-module generated by oriented diagrams in $F$, up to isotopy and the second and third Reidemeister moves, and modulo the skein relations

\bc\begin{enumerate}
\item[(i)] \qquad{$\Xor\  -\ \Yor \qquad =\qquad{(s-s^{-1})}\quad\ \Ior \ ,$}
\vspace*{2mm}
\item[(ii)] \qquad
{$ \Rcurlor \quad=\quad {v^{-1}}\quad \Idor\ ,\qquad \Lcurlor \quad=\quad {v}\quad \Idor. $}
\item[(iii)]\qquad
\unknot \quad = \quad $\frac{v^{-1}-v}{s-s^{-1}}\in\Lambda$.
\end{enumerate}
\ec

Notice our sign conventions for skein relations are different from those of the introduction, in order to match the literature.
We will only be interested in the following cases of this construction.

In the case where $F$ is a rectangle with $m$ inputs at the bottom and $m$ outputs at the top,
the Homfly skein $\calH_m$ carries a product given by stacking rectangles; as an algebra $\calH_m$ can be identified with
the type $A$ Hecke algebra $H_m(z)$ with coefficient ring $\Lambda$.

When $F$ is an annulus with no boundary points, we denote the corresponding Homfly skein by $\calC$.  It is a commutative algebra with product obtained by placing one annulus inside another.  There is a natural map of $\Lambda$-modules from 
$$\widehat{\cdot}: \calH_m \rightarrow \calC$$ 
given by sending a braid $T$ to its closure $\widehat{T}$.
The image of the closure map is $\calC_m$; the submodule $\calC_+$ generated by $\cup_{m \geq 0} \calC_m$ is a subalgebra of $\calC$, generated by diagrams for which all strands are oriented counter-clockwise.

It follows from work of Turaev \cite{turaev} that there exists a grading-preserving isomorphism between $\calC_+$ and the ring of symmetric functions in infinitely many variables, with coefficients in $\Lambda$.  This identification is very useful for the skein-theoretic computations in \cite{lukac-morton, morton-manchon}.  
For each partition $\lambda \vdash m$, there exists an associated idempotent element $e_\lambda \in \calH_m$ constructed by Gyoja \cite{gyoja} and studied by Aiston-Morton \cite{aiston-morton}, deforming the classical idempotents of the group algebra of the symmetric group $\Sigma_m$ obtained from Young symmetrizers.
We denote their closures by
$$Q_{\lambda} = \widehat{e_{\lambda}} \in \calC_m;$$
it is proven in \cite{lukac, aiston-morton} that $Q_{\lambda}$ is sent to the Schur function $s_{\lambda}$ under the identification with symmetric functions.

Finally, when $F = \RR^2$, the skein of $F$ is just the ring of scalars $\Lambda$.  In particular, by embedding the annulus into $\RR^2$, we obtain a trace map
$$\langle \rangle: \calC_+ \rightarrow \Lambda.$$

\subsection{Colored HOMFLY polynomials}\label{coloredhomfly}

 Let $\cL$ be a framed link with $r$ labelled components, and let $Q_1, \dots, Q_r$ be diagrams in the skein model of the 
 annulus, with counterclockwise orientation.  The satellite link $$\cL*(Q_1, \dots, Q_r)$$
 is obtained by drawing the diagram $Q_i$ on the annular neighborhood of the $i$-th strand of $\cL$ determined by the framing.  
 Since each $Q_i$ has a natural blackboard framing as a diagram on the annulus, the satellite link inherits a framing as well.
 
 If $\cL$ arises from a diagram in the annulus (with the blackboard framing and counterclockwise orientation), then the class of the satellite link in $\calC_+$ only depends 
 on the equivalence classes of the decorations $[Q_1], \dots, [Q_r] \in \calC_+$.  By construction,
 this class behaves linearly with respect to the decorations.
 
By taking the trace,
 we define the framed HOMFLY polynomial
 $$\langle \cL*(Q_1, \dots, Q_r) \rangle \in \Lambda$$
for any $Q_i \in \calC_+$ which again will behave linearly with respect to $Q_i$.

In particular, if we have $r$ partitions $\arrlam = (\lambda_{1}, \dots, \lambda_{r})$, the satellite element
$$\cL*(Q_{\arrlam}) = \cL*\left(Q_{\lambda_{1}}, \dots, Q_{\lambda_{r}}\right)$$
is formed by decorating the strands of $\cL$ with the idempotent closures
$Q_{\lambda_{i}}$.
After adding an explicit framing factor, the colored HOMFLY polynomial of $\cL$ is defined by
$$\sW(\cL, \arrlam, v,s) = s^{-\sum_{i=1}^{r}\writhe(\cL_i)\kappa_{\lambda_{i}}} v^{\sum_{i=1}^{r}\writhe(\cL_i)|\lambda_i|}
\langle \cL*(Q_{\arrlam})\rangle \in \Lambda.$$

In the above expression, given a component $\cL_i$ of the link, $\writhe(\cL_i)$ denotes the writhe of the component, i.e. the number of
crossings, weighted by sign.  Given a partition $\lambda$, 
$|\lambda|$ is the sum of the parts and we set
$$\kappa_{\lambda} = \sum_{j} \lambda^{(j)}(\lambda^{(j)} - 2j + 1) = 2 \sum_{\Box \in \lambda} c(\Box)$$
where $c(\Box)$ is the content of a box.

The monomial factor ensures that the colored HOMFLY polynomial is independent of the framing of $\cL$.  
It is shown in \cite{lukac-thesis} that this definition agrees with the definition via quantum groups (as given, for instance, in \cite{lin-zheng}).  
Furthermore, it also proven there that if $\lambda_i = (1)$ for all $i$, we recover the HOMFLY polynomial $\sP(L; v,s)$ by the relation
$$\sP(L; v, s)= (- 1)^{|L|} v^{2\mathrm{lk}(L)} \sW(L, (1), \dots, (1); v, s)$$
where $|L|$ is the number of components of $L$ and $\mathrm{lk}(L)$ is the total linking number of $L$.

\subsection{Torus and splice calculations}\label{torussplice}

As we shall explain in Section $4$, algebraic links
can be constructed by iterating the satellite construction described above, starting from a few basic diagrams in the annulus.

By linearity,
in order to calculate their colored HOMFLY polynomials,
it suffices to take one of these basic diagrams and calculate the element in $\calC_+$ obtained
by decorating its strands with idempotent closures $Q_{\lambda}$, expanded as a linear combination of basis elements $Q_{\nu}$.  
The exact calculations can all be written in terms
of standard operations on symmetric functions.

The first diagram we need is the torus link.  Given a pair of relatively prime positive integers $(m,n)$, 
let $\beta_{m}^{n}\in \calH_m$ denote the $n$-th power of the braid $\beta_m$ with $m$ strands:
\begin{figure*}[htp]
\centering
\includegraphics[scale=0.50]{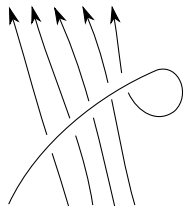}
\caption{$\beta_{m}$ for $m = 5$ }
\end{figure*}

Let $$T_{m}^{n}(Q_\lambda) \in \calC_+$$
denote the element obtained by decorating $T_{m}^{n} = \widehat{\beta_{m}^{n}}$ with $Q_{\lambda}$.

In order to calculate its coefficients when expanded in the basis $Q_{\nu}$, we first recall two operations on the ring of symmetric functions.
First, given a symmetric function
$f(z_1, z_2,\dots)$ and a positive integer $m$, we define the plethysm $f[p_m]$ to be the symmetric function
$$f[p_m](z_1,z_2,\dots) = f(z_1^{m}, z_2^{m},\dots).$$
Given $\lambda$, we define $Q_{\lambda}[p_m] \in \calC_+$ to be the element of the annulus skein corresponding to 
the symmetric function $s_{\lambda}[p_m]$.  Notice that, when we expand $Q_{\lambda}[p_m]$ in
terms of the basis $\{Q_{\mu}\}$, the coefficients are integers rather than rational functions.

The other operation is the framing operator 
$$\tau: \calC_{+} \rightarrow \calC_{+}$$
which is defined by the satellite diagram
$$X \mapsto T_{1}^{1}*(X).$$
Geometrically, this corresponds to applying a full twist to the strands of $X$.
By \cite{aiston-morton}, it
has eigenbasis $Q_{\lambda}$ with eigenvalues given by
$$\tau(Q_{\lambda}) = v^{-|\lambda|}s^{\kappa_{\lambda}} Q_{\lambda}.$$

In these terms, the evaluation of the decorated torus link is given by the following formula, proven in \cite{morton-manchon}:
\begin{lem}[Theorem $13$, \cite{morton-manchon}]\label{torusformula}
$$T_{m}^{n}(Q_\lambda)=\tau^{n/m}(Q_{\lambda}[p_{m}]).$$
\end{lem}

If $m$ and $n$ have common factor $d$, then $T_{m}^{n}$ is the satellite link
$T_{m/d}^{n/d}(\widehat{\id_{d}})$ where $\id_d$ denotes the identity braid in $\calH_d$.
In particular, it has $d$ components and we can reduce the corresponding satellite invariants to the previous lemma
by the identity
$$T_{m}^{n}(Q_1, \dots, Q_d) = T_{m/d}^{n/d}(\prod_{i=1}^{d} Q_i)$$
where the product on the right-hand side is multiplication in $\calC_+$.

%%%%%%%%%
The second diagram we need allows us to splice together a torus knot and another diagram.  Again, given $(m,n)$ relatively prime,
let $\sigma_{m}^{n}\in \calH_{m+1}$ denote the $n$-th power of the braid $\sigma_m$, obtained by adding 
an extra strand to $\beta_m$; see figure $2$.
\begin{figure*}[htp]
\centering
\includegraphics[scale=0.50]{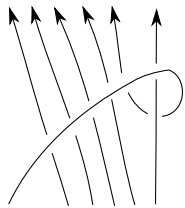}
\caption{$\sigma_{m}$ for $m=5$}
\end{figure*}
The corresponding braid closure corresponds to taking a torus knot on the surface of the torus along with an interior circle.

Let $S_{m}^{n}*(Q_{\lambda}, Q_{\mu}) \in \calC_+$ denote the result of decorating
$\widehat{\sigma_{m}^{n}}$ with $Q_{\lambda}$ and $Q_{\mu}$, so that the latter is on the added strand.

We will prove the following lemma in an appendix; its proof is completely analogous to the torus link case.

\begin{lem}\label{spliceformula}
$$S_{m}^{n}*(Q_{\lambda}, Q_{\mu})= v^{n/m|\mu|}s^{-n/m\kappa_{\mu}}\tau^{n/m}(Q_{\lambda}[p_m]\cdot Q_{\mu}).$$
\end{lem}

One corollary of this is the following expression. 
Given positive integers $m$ and $n$ such that $n \geq m$, we have the identity
\begin{equation}\label{framedsplice}
S_{m}^{n}(Q_{\lambda}, \tau Q_{\mu}) = \tau(S_{m}^{n-m}(Q_{\lambda}, Q_{\mu})).
\end{equation}
This can also be seen directly by an isotopy of diagrams.

\subsection{Unknot and Hopf calculations}\label{unknothopfcalc}

In this section, we state the HOMFLY polynomials for the colored unknot and Hopf link, worked out in \cite{lukac-morton}.

\begin{prop}[\cite{lukac-morton}, Equation $(12)$]\label{unknotformula}
$$\langle Q_{\lambda} \rangle = \prod_{\Box \in \lambda} \frac{v^{-1}s^{c(\Box)} - v s^{-c(\Box)}}{s^{h(\Box)} - s^{-h(\Box)}}$$
\end{prop}

For the colored Hopf link, we can proceed as follows. Given $X \in \calC_+$, we define the meridian operator $\sM_{X}$ 
on $\calC_{+}$ 
whose value $\sM_{X}(Y)$ on $Y \in \calC_{+}$ is defined by the satellite construction shown in figure $3$.
\begin{figure*}[htp]
\centering
%\input{mer8.eps_tex}
%\psfrag{X}{$X$}
\includegraphics[scale=0.50]{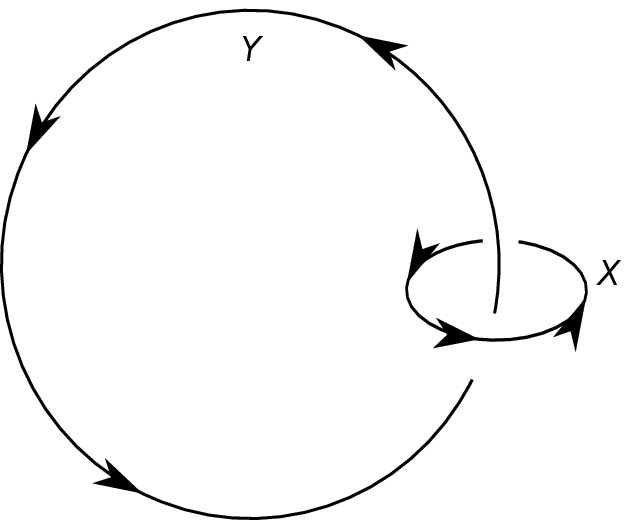}
\caption{$\sM_{X}(Y)$}
\end{figure*}

It is shown in \cite{lukac-morton} that $Q_{\mu}$ is an eigenvector for the meridian operator, i.e.
$$\sM_{X}(Q_\mu) = t_{\mu}(X) Q_{\mu}$$
and that $t_{\mu}$ is a ring homomorphism from $\calC_{+}$ to $\Lambda$.
In these terms, the HOMFLY polynomial for the Hopf link, colored by $\lambda$ and $\mu$ is given by 
$$t_{\mu}(Q_{\lambda}) \langle Q_{\mu} \rangle = t_{\lambda}(Q_{\mu}) \langle Q_{\lambda} \rangle.$$

Although we only need the lowest degree part, we state here the full formula for $t_{\mu}$ given in Lukac-Morton.

Given a partition $\mu$ and a power series $$E(t) = \sum E_{k} t^{k}$$ with coefficients in $\Lambda$,
we define
$$s_{\lambda}(E(t))\in \Lambda$$
to be the element in $\Lambda$ obtained by writing the Schur function $s_{\lambda}(z)$ as a polynomial in elementary symmetric functions $e_{k}(z)$ via the Jacobi-Trudy identity
and making the substitution $e_{k}(z) = E_{k}$.

Take the generating function 
$$E_{\mu}(t) = \prod_{j=1}^{\ell(\mu)} \frac{1 + v^{-1}s^{2\mu_{j}-2j+1}t}{1+ v^{-1}s^{-2j+1}t} \prod_{i \geq 0} \frac{1+vs^{2i+1}t}{1 + v^{-1}s^{2i+1}t}.$$

\begin{prop}[\cite{lukac-morton}, Theorem $4.4$]\label{hopfformula}
We have
$$t_{\mu}(Q_{\lambda}) = s_{\lambda}(E_{\mu}(t)).$$
\end{prop}

\subsection{Lowest degree term}\label{lowdegree}

In this section, we take the formulas from the last section and extract the terms of lowest degree in $v$.

For the colored unknot, it follows from Proposition \ref{unknotformula}
that
$$\langle Q_{\lambda} \rangle = v^{-|\lambda|} \prod_{\Box \in \lambda} \frac{s^{c(\Box)}}{s^{h(\Box)} - s^{-h(\Box)}} + v \sO(v)$$
where $\sO(v)$ denotes elements of $\Lambda$ with no poles at $v = 0$.
We set 
$$\langle Q_{\lambda} \rangle^{\low} =  \prod_{\Box \in \lambda} \frac{s^{c(\Box)}}{s^{h(\Box)} - s^{-h(\Box)}}$$ to be the coefficient 
of $v^{-|\lambda}$.  
After the substitution $q = s^2$, we can rewrite this coefficient in terms of Schur functions via the evaluation
\begin{equation*}
s_{\lambda}(q^{\rho})
\end{equation*}
where $q^{\rho}$ denotes the substitution $z_i = - i + 1/2$ for $i = 1, 2, \dots$.

If we expand it at $s = 0$, we have
$$\langle Q_\lambda \rangle^{\low} = (-1)^{|\lambda|} s^{h(\lambda) + \frac{1}{2}\kappa_{\lambda}} + s\sO(s)$$
where $\sO(s)$ denotes a rational function in $s$ with no poles at $s=0$.

%$$-2 \lambda^{\tr}\cdot \rho = -2 \sum_{i \geq 1} \lambda^{(i),\tr} \cdot (- i + 1/2) = \sum_{\Box\in \lambda} c(\Box) + h(\Box).$$

More generally,
given an element $$X = \sum_{\gamma\sideperp m} c_{\gamma}(v,s) Q_{\gamma}\in \calC_{m},$$
let $$A = \operatorname{min}_{\gamma} \mathrm{ord}_{v=0} c_{\gamma}(v,s)$$
be the degree of the lowest monomial in $v$ that occurs in the coefficients $c_{\gamma}(v,s)$.
We define
$$\langle X \rangle^{\low} = v^{m - A} \langle X \rangle|_{v = 0}.$$
By construction, we have
$$\langle X \rangle = v^{A - m} \left(\langle X \rangle^{\low} + v\cdot \mathsf{O}(v)\right)$$
where $\sO(v)$ again denotes terms regular at $v=0$.
Note that if $\langle X \rangle^{\low}\ne 0$ then
it is the term of lowest $v$-degree in $\langle X \rangle$, although
it could vanish in general.

For meridian operators,
we extract the lowest degree contribution from $t_{\mu}$:
$$t_{\mu}(Q_{\lambda}) = v^{-|\lambda|} t_{\mu}^{\low}(Q_{\lambda}) + v\sO(v).$$
From Proposition \ref{hopfformula}, we can deduce the following formula, again via Schur function evaluations.
After the substitution $q = s^2$, we have
\begin{equation}\label{hopflowdegree}
t_{\mu}^{\low}(Q_{\lambda}) = s_{\lambda}( q^{\mu + \rho})
\end{equation}
where $q^{\mu + \rho}$ denotes the substitution $z_i = q^{\mu^{(i)} - i + 1/2}$ for $ i = 1, 2, \dots$.
To see this, use the standard involution $\omega$ on symmetric functions which exchanges elementary and complete symmetric functions.

By expanding at $q=s=0$, we obtain the expansion around $s=0$ of the lowest $v$-degree term:
\begin{align}\label{hopflowestorder}
\langle \sM_{\mu}(Q_{\lambda}) \rangle^{\low} &= 
t_{\mu}^{\low}(Q_{\lambda}) \langle Q_{\mu} \rangle^{\low}\\
&= (-1)^{|\lambda| + |\mu|} s^{-2(\lambda^{\tr},\mu^{\tr}) +h(\lambda) + \frac{1}{2}\kappa_{\lambda}+ h(\mu) + \frac{1}{2}\kappa_{\mu}}\left(1 + s\sO(s)\right)\,.\nonumber
\end{align}

%We define the operator $\sM^{\low}_{\lambda}$ on $\calC_+$
%by the formula
%$$\sM^{\low}_{\lambda}(Q_{\mu}) = t_{\mu}^{\low}(Q_{\lambda}) Q_{\mu}.$$
%By definition, it acts on $\calC_{m}$, and preserves elements $X$ of the form \eqref{purevform}.
%Therefore, it makes sense to consider 
%$$\langle \sM^{\low}_{\lambda}(X) \rangle^{\low}.$$
%which satisfies
%$$\langle \sM_{\lambda}(X)\rangle = v^{A+ |\lambda| - m} \left(\langle \sM^{\low}_{\lambda}(X)\rangle^{\low} + v\cdot\mathsf{O}(v)\right).$$

Suppose we have elements 
\begin{equation}\label{purevform}
X = v^{A} \sum_{\mu \sideperp m} c_{\mu}(s) Q_{\mu}, \quad Y = v^{B} \sum_{\nu \sideperp n} c_{\nu}(s)Q_{\nu}.
\end{equation}
where every coefficient is a monomial in $v$ of the same degree.
By Corollary $1.2$ of \cite{lukac-morton}, and taking the lowest order terms, we have that
$$\langle \sM_{\lambda}(XY) \rangle^{\low}\langle Q_{\lambda} \rangle^{\low} = \langle \sM_{\lambda}(X) \rangle^{\low} \langle \sM_{\lambda}(Y) \rangle^{\low}\,.$$
More generally, 
this yields the following lemma given $X_1, \dots, X_r$ of the form \eqref{purevform}:
\begin{lem}\label{productrule}
$$
\langle \sM_{\lambda}(\prod_{i=1}^{r} X_i) \rangle^{\low} \langle Q_{\lambda} \rangle^{\low}= \left(\langle Q_{\lambda}\rangle^{\low}\right)^{2-r}\prod_{i=1}^{r}\langle \sM_{\lambda}(X_i) \rangle^{\low}.$$
\end{lem}

\section{Algebraic links} 

In this section, we recall from \cite{eisenbud-neumann}, \cite{neumann-kahler}, and \cite{ctcwall},
the presentation of algebraic links in terms of the constructions from the last section.  We also discuss the
lowest degree behavior of their HOMFLY polynomials, which will be the most important feature for our application, as well
as how they change with respect to blowup.

\subsection{Description of links}\label{linkdescription}

Given a plane curve $$C = \{f(x,y)= 0\} \subset \CC^2$$
with singularity at $p = (0,0)$,
the algebraic link $\cL_{C,p}$
can be described via an iterated satellite construction using the diagrams listed above.
A more detailed discussion can be found in Appendix A of \cite{eisenbud-neumann} 
as well as the first few sections of \cite{neumann-kahler}.  

We first discuss the case where the germ $\widehat{C}$ of the curve singularity at $p$ is irreducible, in which case
$\cL_{C,p}$ is a knot.  
We first assume that $f(x,y) \ne x$, and by solving for $y$ as a function of $x$ via Newton's method, we obtain a Puiseux series
expansion for $y(x)$.  That is, we have a Puiseux series 
$$y(x) = x^{\frac{q_{0}}{p_{0}}}(a_0 + x^{\frac{q_{1}}{p_{0}p_{1}}}(a_1 + x^{\frac{q_{2}}{p_{0}p_{1}p_{2}}}(a_2+ \dots)))$$
such that
$$f\left(x,y(x)\right) = 0.$$
Here, each Newton pair $(p_i, q_i)$ is a pair of relatively prime positive integers for which $p_k = 1$ eventually, and $a_i$ is nonzero.  The topology of $\cL_{C,p}$ only depends on the singularity type of $\widehat{C}$ so is unchanged if we truncate $y(x)$ after some finite number $s$ of iterations.  In our notation,
the Newton pairs depend on the choice of coordinates (and truncation).

In these terms, $\cL_{C,p}$ is an iterated torus knot, with parameters given by the Newton pairs $(p_i,q_i)$.  More precisely, $\cL_{C,p}$ can be embedded as a diagram $L_{C}$ in the annulus so that
\begin{equation}\label{iteratedtorus}
L_{C} = T_{p_{0}}^{q_{0}}*( T_{p_{1}}^{q_{1}}*(\dots(T_{p_{s}}^{q_{s}})\dots).
\end{equation}
In the degenerate case $y(x) = 0$,
we set $L_C = T_{1}^0$, i.e. the closure of a single unknotted strand.

\begin{remark}
We make a few remarks about this formula.  First, we use the framing inherited from the annulus, in which a torus knot is equipped with the framing coming from the surface of the torus; this differs from the cabling convention in \cite{oblomkov-shende}, for instance.  This convention is explained in the references (see, e.g. footnote 1 in \cite{neumann-kahler}).
Second, in some of the references (e.g. \cite{ctcwall}), it is convenient to assume $q_0 \geq p_0$, but this formula is valid without that assumption.  Finally,
while the topology of the link $\cL_C$ only depends on the singularity, the annulus diagram $L_C$ we construct will depend on the choice of coordinates.
\end{remark}
%%% mention that we allow q_0< p_0

In the reducible case, let $\widehat{C}_1, \dots, \widehat{C}_{r}$ denote the branches of $\widehat{C}$;
for each $1\leq i \leq r$,
we have an associated Puiseux expansion as above:
\begin{equation}\label{puiseuxbranch}
y_{i}(x) = x^{\frac{q^{(i)}}{p^{(i)}}}\left(a_{i} + z_{i}(x^{\frac{1}{p^{(i)}}})\right)
\end{equation}
where we only write out the first step of the iterated expansion.
We again assume for now that $x=0$ is not a branch.
The link component for this branch is constructed as above; to explain how these components are linked, we use the 
splice diagram $S_{m}^{n}$ from Section \ref{torussplice}.
As before, we can assume each Puiseux series is finite by truncating after a sufficiently large number of steps, without affecting the topology of the link (see, e.g., Theorem $1.1$ in \cite{neumann-kahler}), and we can construct the link inductively.

For each index $i$, we associate the pair 
$$\left(\alpha_i = \frac{q^{(i)}}{p^{(i)}}, a_{i}\right)\in \QQ_{>0}\times \CC.$$
In the case of $y(x) = 0$, we associate the pair $(\infty, 0)$
For each pair $(\alpha, a)$, 
let 
$$\{i_{0}, \dots, i_{n}\}$$
denote the set of indices $i$ which have $(\alpha, a)$ as its associated pair.
If we consider the set of Puiseux series
$$\{z_{i_{0}}(x), \dots, z_{i_{n}}(x)\},$$
we inductively assume we have defined an annulus diagram
$$L_{(\alpha, a)}.$$

For each $\alpha \in \QQ_{> 0} \cup \infty$, we set
$$L_{\alpha} = \prod_{a} L_{(\alpha,a)}$$
where $\prod$ denotes concatenation of annuli.  The order of the product does not matter up to isotopy.
% --- see Lukac.

After relabelling, we can assume that $\alpha_1 < \alpha_2 \dots < \alpha_{k}$ and that every $\alpha_i$ occurs in this initial sequence.

The link $\cL_C$ can be represented by the annulus diagram $L_C$ where
\begin{equation}\label{algebraiclinkformula}
L_C = S_{p^{(1)}}^{q^{(1)}}*\left(L_{\alpha_{1}},S_{p^{(2)}}^{q^{(2)}}*(L_{\alpha_{2}}, \dots, 
S_{p^{(k-1)}}^{q^{(k-1)}}*(L_{\alpha_{k-1}},T_{p^{(k)}}^{q^{(k)}}*(L_{\alpha_{k}})))\dots\right).
\end{equation}
If $\alpha_k = \infty$, then we have the same expression, but with 
$T_{p^{(k)}}^{q^{(k)}}*(L_{\alpha_{k}})$ replaced with $L_{\alpha_{k}}$.

Finally, if we add the branch $x=0$ to $\widehat{C}$, 
then the new link is obtained from $L_C$ by adding a meridian around the annulus, as in figure $3$.

\subsection{Lowest degree terms}\label{algebraiclowdegree}

Let $(C,p)$ be as in the last section, and we add the condition that $C$ does not contain $x=0$ as a branch.  
Fix truncated Puiseux series for the branches
$C_{1},\dots, C_{r}$ at $p$; let $L_C$ be the corresponding annulus diagram determined above.
Fix partitions $\arrmu = (\mu_1, \dots, \mu_r)$ and consider the
satellite diagram $L_C*(Q_{\arrmu}).$
If we take its class in $\calC_+$ and expand in the basis $Q_{\nu}$, we will calculate the lowest order term in $v$ and $s$ as follows.

We first recall some lemmas regarding symmetric functions.
Given two partitions $\mu$ and $\nu$, let $\mu \cup \nu$ denote the partition obtained by concatenating parts,
e.g. $(3,2) \cup (1,1) = (3,2,1,1)$.  Similarly, given a natural number $k$, $\mu^{\cup k}$ denotes the partition obtained by concatenating $\mu$ with itself $k$ times.
If we take either the plethysm of a Schur function or the product of two Schur functions, then if we expand the result in Schur functions, the following lemma describes the lowest nonzero term that occurs with respect to the dominance ordering $>$ on partitions

\begin{lem}\label{schurexpansion}   
$$s_{\mu}[p_k](z) = \pm s_{\mu^{\cup k}}(z) + \sum_{\rho > \mu^{\cup k}} c_{\rho} s_{\rho}(z)$$
and
$$s_{\mu}(z)\cdot s_{\nu}(z) = s_{\mu \cup \nu}(z) + \sum_{\rho > \mu \cup \nu} \mathrm{LR}^{\rho}_{\mu,\nu} s_{\rho}(z),$$
where $\mathrm{LR}$ denotes Littlewood-Richardson coefficients.
\end{lem}

This yields the following corollary.  In what follows, fix partitions $\mu$ and $\nu$ as well
as relatively prime integers $m,n$. 
Let $\lambda = \mu^{\cup m}\cup \nu$.

\begin{cor}
$$[S_{m}^{n}*(Q_{\mu}, Q_{\nu})] = v^{-n|\mu|} s^{\frac{n}{m}(\kappa_{\lambda} - \kappa_{\nu})}\left(\pm Q_{\lambda} + \sum_{\rho> \lambda} c_{\rho}(s) Q_{\rho}\right)$$
where $c_{\rho}(s)$ has no poles at $s=0$.
\end{cor}

By iterating this relation,
we recover a similar statement for the general case using the following lemma.

\begin{lem}
The expression 
$$\kappa_{\lambda \cup \mu} - \kappa_{\mu}$$
is weakly increasing for fixed $\lambda$ as $\mu$ increases with respect to the dominance ordering, and 
similarly for fixed $\mu$ and $\lambda$ increases.
\end{lem} 
\begin{proof}
This follows from the formula
$$\kappa_{\lambda \cup \mu} - \kappa_{\mu} = \kappa_{\lambda} - 2 (\lambda^{\tr}, \mu^{\tr}).$$
\end{proof}

This lemma implies that as we iterate the satellite construction with the diagram $S_{m}^{n}$, and look at
the lowest order of vanishing at $s=0$ for nonzero coefficients of $Q_{\gamma}$, this bound is always achieved
on the partition obtained by concatenating parts.

 Let $m_i$ denote the number of strands in the annulus diagram for each connected component of $L_C$; equivalently, this is the intersection multiplicity of $\widehat{C_{i}}$ with the curve $x = 0$, since both are the linking number of the knot with the meridian.
Let
$$\gamma = \mu_{1}^{\cup m_{1}} \cup \dots \mu_{r}^{\cup m_{r}}$$
denote the concatenation of $\arrmu$ with multiplicities $m_i$.

\begin{prop}\label{lowdegreealgebraic}
There exist exponents $A$ and $B$ and a choice of sign such that
$$[L_{C}*(Q_{\mu_{1}}, \dots, Q_{\mu_{r}})]
= v^{A}s^{B}\left(\pm Q_{\gamma} + \sum_{\gamma' > \gamma} c_{\gamma'}(s) Q_{\gamma'}\right)$$
where $c_{\gamma'}(s)$ has no poles at $s=0$.
\end{prop}

We can calculate the exponents $A$ and $B$ recursively in terms of the Puiseux series.  We will only do so for $A$.  If
$C$ has a single branch with Newton pairs $\{(p_0, q_0), \dots, (p_s,q_s)\}$, then
$$A = -|\lambda|\left(q_{s} + q_{s-1}p_{s} + q_{s-2}p_{s-1}p_{s} + \dots q_{0}p_{1}\dots p_{s}\right).$$
Similarly, if we have 
multiple components, we take the above expression for each component and add them.

If $x=0$ is a component of $C$ decorated with partition $\lambda$, we can calculate $L_C$ from Proposition \ref{lowdegreealgebraic} by applying the 
meridian operator $\sM_{\lambda}$.  Note that its coefficients in the basis $Q_{\mu}$ are no longer homogeneous with respect to the variable $v$.

\subsection{Blowup analysis}\label{blowupanalysis}

In this section, we study how the link $\cL_C$ behaves with respect to blowing
up at $p \in C$.  
We fix notation as in Section \ref{blowupsetup}, so we have irreducible components $C_1, \dots, C_r$ each of which has a unique branch through $p$ by assumption.
Let $C'_1, \dots, C'_r$ denote the proper transform of each of these components; by assumption, each $C'_i$ meets the exceptional divisor
$E$ at a unique point in $C' \cap E = \{p_1, \dots, p_e\}$.  For $k= 1, \dots, e$, let $B_k$ denote the planar singularity of $C'$ at $p_k$ let $D_k$ denote the planar singularity of $C' \cup E$ at $p_k$.
We wish to relate the algebraic links associated to $C$, $B_k$ and $D_k$.

By appropriate choice of coordinates $(x,y)$, we will assume that $\{x=0\}$ and $\{y=0\}$ are not branches.
Choose truncated Puiseux expansions for each $C_i$ at $p$, as in equation \eqref{puiseuxbranch}:
$$y_{i}(x) = x^{\frac{q^{(i)}}{p^{(i)}}}(a_{i} + z_{i}(x^{\frac{1}{p^{(i)}}})).$$
By a further choice of coordinates, we can assume that 
$$\alpha_i = \frac{q^{(i)}}{p^{(i)}}\geq 1,\quad\textup{ for }i = 1, \dots, r.$$
If we blow up at the point $(0,0)$, we work on the affine chart $U$ with coordinates $(x,w)$ of the blowup corresponding to the 
map
$$(x,w) \mapsto (x, y = xw).$$
On $U$, the divisor $E$ is given by the equation $x= 0$.
By the condition on $\alpha_i$, we know that $C' \cap E \subset U$.

We first suppose that $\alpha_i > 1$ for all $i$.
For each $i$, we can calculate the Puiseux expansion for $C'_i$ by
making the substitution $y = xw$ in the Puiseux expansion for $C_i$.  The result is
$$w_{i}(x) = x^{\frac{q^{(i)}-p^{(i)}}{p^{(i)}}}(a_{i} + z_{i}(x^{\frac{1}{p^{(i)}}})).$$
In particular, $C'_i \cap E = (0,0)$ for all $i$ and we can study the link $L_{C'}$ of the singularity of $C'$ at $(0,0)$

For each component, we see that only the first Newton pair for the series is changed, while the rest remain the same.  Therefore, 
using equation \eqref{iteratedtorus}, we have the relation between annulus diagrams
$$[L_{C_{i}}] = \tau[L_{C'_{i}}] \in \calC_+$$
for the link components.
To understand the full link $L_C$ in terms of $L_{C'}$, we use equation \eqref{algebraiclinkformula}.
Using either a direct geometric isotopy, or
 the framing formula
$$S^{q}_{p}(X, \tau Y) = \tau S^{q-p}_{p}(X, Y)$$
 for the splice link $S_{p}^{q}$, we see
 that 
 \begin{equation}\label{preliminaryknotblowup}
 [L_{C}] = \tau [L_{C'}]
 \end{equation}
with the same statement if we decorate $L_C$ with a partition $\mu$.

If there are $i$ such that $\alpha_i = 1$, we proceed as follows.
For each $i$ such that $\alpha_i = 1$ (so $p^{(i)} = q^{(i)}= 1$),
the branch for $C'_i$ is given by the series expansion
$$w_i(x) = a_{i} + z_{i}(x^{\frac{1}{p^{(i)}}}) = a_{i} + z_{i}(x).$$
In particular, 
$$C'_i \cap E = (0, a_i) \in U$$ 
is determined by the leading coefficient $a_i$.

There is an intersection point $(0,a) \in C' \cap E$ for each $a\ne 0$ such that the pair $(\alpha = 1,a)$ is associated to a branch, plus the origin
$(0,0) \in C' \cap E$ if any $\alpha_i > 1$. 

If we change to coordinates based at $p_k = (0,a_k)$, via the coordinate change $x' = x, y' = y - a_k$, then the link for the singularity $B_k$ of $C'$ at $p_k$ is given by the formula
$$[L_{B_{k}}] = [L_{(1, a_{k})}]$$
where $[L_{(1,a_{k})}]$ is the skein element defined inductively using $z_{i}(x)$ in Section \ref{linkdescription}.
If we combine this with equation \eqref{algebraiclinkformula} and our earlier analysis \eqref{preliminaryknotblowup} of the branches with $\alpha > 1$, this yields the following answer for the general case.
In what follows, we assume we have labelled $\{p_1, \dots, p_e\}$ so that $p_e = (0,0)$ if any $\alpha_i > 1$.

\begin{prop}\label{blowupforknots}
The element $[L_C]\in \calC_+$ is given by the formula
$$[L_{C}] = S^{1}_{1}*\left(L_{B_{1}}\cdot \dots \cdot L_{B_{e-1}}, \tau L_{B_{e}}\right)$$
if any $\alpha_i > 1$; otherwise it is given by
$$[L_{C}] = S^{1}_{1}*\left(L_{B_{1}}\cdot \dots \cdot L_{B_{e}}, \emptyset\right) = T^{1}_{1}*\left(L_{B_{1}}\cdot \dots \cdot L_{B_{e}}\right).$$
The analogous statement holds if the irreducible components of $B$ are decorated by partitions $\mu_1, \dots, \mu_r$.
\end{prop}

Finally, since the exceptional divisor is given by the equation $x=0$, if we add it to $B_k$, decorated with partition $\lambda$, we have the following.
\begin{prop}\label{meridianblowup}
$$[L_{D_{k}}] = \sM_{\lambda}(L_{B_{k}}).$$
The analogous statement holds if $B_{k}$ is decorated by partitions $\arrmu[k]$.
\end{prop}

\section{Blowup identity for links}

We prove a parallel version of Proposition \ref{realblowupidentity} for links.  
In order to prove Theorem \ref{mainthm2}, it will ultimately suffice to prove compatibility of these two identities, which we do in the next section.
The argument here is largely independent of the preceding sections (other than
definitions of section \ref{skeinrecap} and the calculations of section \ref{unknothopfcalc}).

\subsection{Vertex flop}

We first write down the identity for the colored unknot.
In this case, Proposition \ref{unknotformula} gives us a closed formula for its HOMFLY polynomial in terms of Schur functions.

One approach to write down the blowup identity is to use the topological vertex to equate this Schur function expression with
the corresponding stable pairs generating series; Proposition \ref{realblowupidentity} then becomes a symmetric function identity which can be written as a blowup formula for the colored unknot.  However, we can skip most of these steps by using the paper \cite{minabe-flop} where they prove the symmetric function identity directly, with the goal of understanding the flop invariance of the topological vertex.

Let 
\begin{align*}
Z_{\mu}(q,Q) &= s_{\mu}(q^{\rho}) \prod_{\Box \in \mu} (1 + Q q^{-c(\Box)})\\
&= q^{\kappa_{\mu}/4} \prod_{\Box \in \mu} \frac{1 + Q q^{-c(\Box)}}{q^{h(\Box)/2} - q^{-h(\Box)/2}}.
\end{align*}
The following identity is a special case of Theorem $2.7$ of \cite{minabe-flop}, with $\lambda_{1} = \mu$ and the other partitions $\emptyset$:
\begin{equation}\label{vertexflop}
Q^{|\mu|} Z_{\mu}(q,Q^{-1}) = \frac{1}{\prod_{k \geq 1} (1 + Q q^{k})^{k}}
\sum_{\lambda} Q^{{|\lambda|}} q^{-\frac{1}{2}(\kappa_{\mu}+\kappa_{\lambda})} s_{\mu}(q^{\rho})s_{\lambda}(q^{\mu + \rho})s_{\lambda}(q^{\rho})\,.
\end{equation}

\subsection{General formula}

Using Proposition \ref{unknotformula}, we see that
\begin{equation*}
Z_{\mu}(q= s^{2},Q= - v^{2}) = v^{|\mu|} \langle Q_{\mu} \rangle
\end{equation*}
after the change of variables (which again is a standard consequence of the relation between the topological vertex and colored Chern-Simons invariants of the unknot).
Similarly, the Schur function evaluations on the right-hand side of \eqref{vertexflop} are precisely the lowest-degree invariants for the Hopf link and the unknot from Section \ref{lowdegree}.
The monomial shift $q^{-\kappa_{\lambda}/2}$ is, after ignoring the $v$-monomial, the eigenvalue of $\tau^{-1}$, the inverse of the framing shift operator.

Therefore, we can rewrite the vertex flop identity as
\begin{equation}\label{knotvertexflop}
(-v)^{|\mu|} \langle Q_{\mu} \rangle |_{v = v^{-1}} = \frac{1}{Z_{Y}(s^{2},-v^{2})} \sum_{\lambda} (-v^{2})^{|\lambda|}
s^{-\kappa_{\lambda}} \langle Q_{\lambda} \rangle^{\low} \langle \sM_{\lambda} \tau^{-1} Q_{\mu} \rangle^{\low}\,.
\end{equation}

Given an element $X \in \calC_{m}$ of the form \eqref{purevform}
$$X = v^{A} \sum_{\gamma \sideperp m} c_{\gamma}(s) Q_{\gamma} \in \calC_{m},$$
the trace $\langle \sM_{\lambda} \tau^{-1} X \rangle^{\low}$
is linear with respect to coefficients $c_{\gamma}(s)$, so we can extend equation \eqref{knotvertexflop} to all $X$ of this form.

\begin{prop}\label{knotblowupidentity}
Given $X \in \calC_{m}$ of the form \eqref{purevform}, we have
\begin{align*}
(-1)^{m}v^{m + A}  \langle X \rangle |_{v = v^{-1}} = \frac{1}{Z_{Y}(s^{2},-v^{2})} \sum_{\lambda} (-v^{2})^{|\lambda|}
s^{-\kappa_{\lambda}} \langle Q_{\lambda} \rangle^{\low} \langle \sM_{\lambda}\tau^{-1} X \rangle^{\low}\,.
\end{align*}
\end{prop}

In particular this applies to the skein elements $[L_C]$ assuming coordinates are such that $\{x=0\}$ is not a branch.
One feature of this identity is that, on the right-hand side, the entire $v$-dependence comes from the auxiliary partition $\lambda$.

Finally, since any link $L$ can be written as a braid closure, this identity can be applied in general, i.e. not just for algebraic links.
In geometric terms, given a braid $Y$ we are calculating the colored HOMFLY polynomial after composing with a full twist (and taking the closure) in terms of the $v=0$ HOMFLY specialization of the link obtained by adding a meridian to $Y$.  These two links can be related by the local Kirby moves of Fenn-Rourke \cite{fenn-rourke}, suggesting a $3$-manifold interpretation of this identity.

\section{Proof of main theorems}

In this section, we give proofs of Theorems \ref{mainthm1} and \ref{mainthm2}.
Given a plane curve singularity $C$ as always, with $r$ irreducible components, and an $r$-tuple
of partitions $\arrmu = (\mu_1, \dots, \mu_r)$, choose local coordinates, truncated Puiseux representatives of the branches, and let
$[L_C] \in \calC_{+}$ be the skein element constructed in Section \ref{linkdescription}.

Since the HOMFLY polynomial differs from the skein trace by a monomial shift, we can rephrase the claim of Theorem \ref{mainthm2} as follows:
\begin{prop}\label{skeinmainprop}
For some choice of exponents $a, b, \epsilon$, we have the identity
\begin{equation}\label{skeinmaintheorem}
\sZ'(Y,C, \arrmu; q, Q) = (-1)^{\epsilon} v^{a} s^{b} \langle [L_C*\left(Q_{\arrmu^{\tr}}\right)] \rangle\,
\end{equation}
after the change of variables
$$q = s^{2},\quad Q = -v^{2}.$$
\end{prop}
Since the left-hand side is of the form
$$1+ q\cdot \sO(q) + Q\cdot \sO(Q,q)$$
where $\sO(q)$ is regular at $q=0$ and $\sO(Q,q)$ is regular at $Q=0$ (but perhaps not at $q=0$),
the unknown exponents are uniquely determined by the behavior of the skein trace at $v, s \rightarrow 0$.  The claim only depends on the underlying link and is independent of the skein presentation (which depends on our choice of coordinates, etc.).

As explained in the introduction, our strategy is to use our blowup identities and embedded resolution of singularities to reduce to the 
case where $C$ is either smooth or a nodal singularity.  In fact, because of the structure of these formulas, our induction only requires
the $v=0$ specialization of \eqref{skeinmaintheorem}.

\subsection{Specialization at $v=0$}

We first state what we should think of as the specialization of Proposition \ref{skeinmainprop} after taking the lowest degree terms in $Q$ and $v$ on each side.
\begin{prop}\label{v=0skeinprop}
\begin{equation}\label{v=0maintheorem}
\sZ'(Y,C, \arrmu; q, Q=0) = (-1)^{\epsilon} s^{b} \langle [L_{C}*\left(Q_{\arrmu^{\tr}}\right)] \rangle^{\low}\,.
\end{equation}
for some choice of $\epsilon$ and $b$ and after the change of variables $q= s^2$.
\end{prop}
A priori, the right-hand side could be zero, so this proposition is not immediately implied by Proposition \ref{skeinmainprop}, although it is once
we show this nonvanishing.  We return to this point at the end of Section \ref{sectionmonomialshift}.
%As shown in section \ref{}, the skein element 
%$$[L_C*\left(Q_{\arrmu^{\tr}}\right)]$$
%is of the form \eqref{}, so we can extract the lowest degree term on the right-hand side via $\langle \rangle^{\low}$ to obtain the $v=0$ specialization:

The first immediate observation is that Proposition \ref{v=0skeinprop} holds in the case when $C = \{x = 0\}$ or $C = \{xy = 0\}$.
\begin{lem}\label{basecase}
Proposition \ref{v=0skeinprop} holds for $C = \{x=0\}$ with any partition $\mu$ and for $C = \{xy = 0\}$ with any pair of partitions $\mu_1, \mu_2$.
\end{lem}
\begin{proof}
This follows from the fact that the two-leg topological vertex calculates both the stable pairs vertex and 
the $v=0$ specialization of the HOMFLY polynomial of the Hopf link.  The relation between the topological vertex
and the box-counting generating function of the (non-virtual) Donaldson-Thomas vertex is proven in \cite{orv}.  The same argument gives 
the stable pairs vertex of \cite{pt-vertex}, as will appear in upcoming work of Ben Young \cite{benyoung}.  In the meantime, a geometric argument for the equality of
the virtual DT and PT vertex is provided in \cite{mpt} using techniques of \cite{moop} (valid even in the non-Calabi-Yau case).  After a global sign, the Euler-characteristic and virtual 
vertices are related by the substitution $q \rightarrow -q$, so this gives equality of topological Euler characteristics as well.
The relation between the two-leg topological vertex and the HOMFLY polynomial at $v=0$ is just equation \eqref{hopflowdegree}.
\end{proof}

\subsection{Inductive step}

We now state the inductive step.  Given $C$ and $\arrmu$ as above, we take the blowup of $\CC^2$ at $p$ and fix all notation for $D_k, p_k,$ etc. as in Section \ref{blowupsetup}.
Assume moreover that we have chosen local coordinates for $C$ (and chosen truncated Puiseux series) satisfying the conditions of Section \ref{blowupanalysis},
so that the results of that section apply and we have the corresponding annulus diagrams $L_C$ and $L_{D_{k}}$.

\begin{prop}\label{inductiveprop}
Suppose Proposition \ref{v=0skeinprop} holds for each planar singularity $(D_k, p_k)$ and $L_{D_{k}}$ labelled with partitions $(\arrmu[k], \lambda)$, for all possible values of the partition $\lambda$.
Then Proposition \ref{skeinmainprop} holds for $C$ labelled with $\arrmu$.
\end{prop}
\begin{proof}

It suffices to 
match $\sZ'(Y,C,\arrmu)$ and $\langle [L_C*(Q_{\arrmu^{\tr}})]\rangle$ after the substitutions $Q = Q^{-1}$ and $v = v^{-1}$ respectively.
To do this, we will use the two blowup identities, given in Proposition \ref{realblowupidentity} and Proposition \ref{knotblowupidentity}.
Since we are allowed to match the two sides of equation \eqref{skeinmaintheorem} after a monomial shift, we need
to show that the right-hand side of these blowup identities agree, up to a possible global sign and monomial shift.

If we compare these two expressions, the denominators are both given by $Z_Y(q,Q)$ after the change of variables, so we restrict to the
summation over all partitions.
Let us match individual terms in each summation, by sending the summand
$$F_{\lambda} (q,Q) = Q^{|\lambda|}q^{f(\lambda,\arrmu) + \delta} \prod_{k=1}^{e} \sZ'(Y, D_{k}, (\arrmu[k], \lambda); q, Q=0) \cdot H_{\lambda}(q)^{2-e}$$
 corresponding to $\lambda$ in Proposition \ref{realblowupidentity} to the summand 
 $$G_{\lambda}(v,s) = (-v^{2})^{|\lambda^{\tr}|}
s^{-\kappa_{\lambda^{\tr}}} \langle Q_{\lambda^{\tr}} \rangle^{\low} \langle \sM_{\lambda^{\tr}}\tau^{-1} L_C*(Q_{\arrmu^{\tr}}) \rangle^{\low}\,$$
corresponding to $\lambda^{\tr}$ in Proposition \ref{knotblowupidentity}.

By Proposition \ref{blowupforknots} and equation \eqref{framedsplice},
we see that
$$\tau^{-1} L_C*(Q_{\arrmu^{\tr}}) = \prod_{k=1}^{e} L_{B_{k}}*(Q_{\arrmu[k]^{\tr}}).$$
Up to a signed monomial shifts, we have that
\begin{align*}
F_{\lambda}(q,Q) &= \textup{monomial}\cdot \prod_{k} \sZ'(Y, D_{k}, (\arrmu[k], \lambda); q, Q=0) \cdot H_{\lambda}(q)^{2-e}\\
&= \textup{monomial} \cdot\prod_{k} \langle L_{D_{k}}*(Q_{\arrmu^{\tr}[k]}, Q_{\lambda^{\tr}}) \rangle^{\low} \cdot  \left(\langle Q_{\lambda^{\tr}} \rangle^{\low}\right)^{2-e}\\
&= \textup{monomial} \cdot \prod_{k} \langle \sM_{\lambda^{\tr}}L_{B_{k}}*(Q_{\arrmu^{\tr}[k]}) \rangle^{\low} \cdot  \left(\langle Q_{\lambda^{\tr}} \rangle^{\low}\right)^{2-e}\\
&= \textup{monomial}\cdot \langle \sM_{\lambda^{\tr}} \tau^{-1} L_C*(Q_{\arrmu^{\tr}}) \rangle^{\low} \langle Q_{\lambda^{\tr}} \rangle^{\low}.
\end{align*}
In the second line, we use the assumption, and the fact $H_{\lambda}$ and $\langle Q_{\lambda} \rangle$ are both given by $s_{\lambda}(q^{\rho})$ up to monomials.
For the third line, we use Proposition \ref{meridianblowup}, and, for the last line, we use Lemma \ref{productrule}.

In summary, for each $\lambda$, there exists $\epsilon(\lambda)$ and $a(\lambda)$ such that
$$F_\lambda(q,Q) = (-1)^{\epsilon(\lambda)} s^{a(\lambda)} G_\lambda(v,s).$$
Notice that the dependence on $Q = -v^{2}$ is the same for $F$ and $G$.

\begin{lem}\label{monomialshift}
The exponents $\epsilon(\lambda)$ and $a(\lambda)$ are independent of $\lambda$.
\end{lem}

Assuming this lemma, we finish the proof of Proposition \ref{inductiveprop}.  Indeed, using this lemma, we see
$$\frac{1}{Z_Y(q,Q)}\sum_{\lambda} F_{\lambda}(q,Q) = (-1)^{\epsilon} s^{a} \frac{1}{Z_{Y}(s^{2},-v^{2})} \sum_{\lambda} G_{\lambda}(v,s)$$
which implies
$$ \sZ'(C, \arrmu; q, Q) = \textup{monomial}\cdot \langle [L_C*(Q_{\arrmu})] \rangle.$$
\end{proof}

\subsection{Proof of Lemma \ref{monomialshift}}\label{sectionmonomialshift}

It remains to prove the lemma.
\begin{proof}
We first handle $a(\lambda)$.
In order to do this, we compare the order of vanishing at $s=0$ of $G_{\lambda}(v,s)$ and show that it equals
$2f(\lambda, \arrmu)$ up to terms that are independent of $\lambda$.

Recall the multiplicities $m'_{i}$ of the intersection $C'_i \cap E$ and let
$\gamma$ be the partition 
$$\gamma = (\mu^{\tr}_{1})^{\cup m'_{1}} \cup \dots \cup (\mu^{\tr}_{r})^{\cup m'_{r}}$$
obtained by concatenating the parts of $\mu^{\tr}_{i}$ with multiplicity $m'_i$.
If we take the transpose partition, we see that
$$\gamma^{\tr} = \sum m'_{i} \mu_{i}$$
where the sum is taken componentwise (i.e. $(\mu+ \nu)^{(j)} = \mu^{(j)} + \nu^{(j)}$).

By Proposition \ref{lowdegreealgebraic} and Lemma \ref{schurexpansion}, we have an expansion
\begin{equation}\label{linkexpansion}
\prod_{k} L_{B_{k}}*(Q_{\arrmu^{\tr}[k]}) = v^{A}s^{B}\left(\pm Q_{\gamma} + \sum_{\nu> \gamma} c_{\nu}(s) Q_{\nu}\right)
\end{equation}
with $c_{\nu}(s)$ regular at $s=0$.

Recall from Section \ref{lowdegree} that the order of vanishing at $s=0$ of
$$\langle \sM_{\lambda^{\tr}} Q_{\nu} \rangle^{\low}$$
is
$$ -2(\lambda,\nu^{\tr}) +h(\lambda^{\tr}) + \frac{1}{2}\kappa_{\lambda^{\tr}}+ h(\nu) + \frac{1}{2}\kappa_{\nu}.$$
For fixed $\lambda$, as a function of $\nu$, this is strictly increasing with respect to the dominance order on partitions; indeed this is true
for $h(\nu) + \frac{1}{2}\kappa_{\nu}$ and the other terms are weakly increasing.
Therefore, only the $Q_{\gamma}$ term in equation \eqref{linkexpansion} contributes to the order of vanishing of $G_{\lambda}(v,s)$.
Adding in the order of vanishing from the other terms, we 
see that this order of vanishing is
\begin{align*}
-\kappa_{\lambda^{\tr}} &+ h(\lambda) + \frac{1}{2}\kappa_{\lambda^{\tr}}\\
&-2(\lambda,\gamma^{\tr}) +h(\lambda^{\tr}) + \frac{1}{2}\kappa_{\lambda^{\tr}}+ h(\gamma) + \frac{1}{2}\kappa_{\gamma}\,.
\end{align*}
Up to terms not involving $\lambda$, this is precisely
$$2f(\lambda, \arrmu)$$
which concludes the proof for $a(\lambda)$.

For $\epsilon(\lambda)$, we proceed similarly; we only need to calculate the sign associated to the $Q_{\gamma}$ term in \eqref{linkexpansion}.  The coefficient of $Q_{\gamma}$ is independent of $\lambda$, so can be ignored.
The sign in front of the lowest order term in $\langle Q_{\lambda^{\tr}} \rangle^{\low}$ 
and $\langle \sM_{\lambda^{\tr}} Q_{\gamma} \rangle$ are
 $(-1)^{|\lambda|}$ and $(-1)^{|\lambda| + |\gamma|}$ respectively,
 so the product is independent of $\lambda$.
 This concludes the proof.
 \end{proof}
 
 As a corollary of this argument, observe that given any skein element $X$
 of the form 
 $$v^{A}s^{B}\left(\pm Q_{\gamma} + \sum_{\nu> \gamma} c_{\nu}(s) Q_{\nu}\right)$$
 with $c_{\nu}(s)$ regular at $s=0$, we have
 $$\langle \sM_{\lambda^{\tr}} X \rangle^{\low} \ne 0,$$
 since it has finite order of vanishing at $s=0$.

Given any planar singularity $C$ and arbitrary choice of coordinates, Puiseux truncations, etc., 
$[L_{C}*\left(Q_{\arrmu^{\tr}}\right)]$ has the form  $\sM_{\lambda^{\tr}}X$ with $X$ satisfying the above condition.  If $\{x=0\}$ is not a branch,
then we just take $\lambda = \emptyset$;otherwise the meridian operator is nontrivial.
Therefore, we deduce
$$\langle [L_{C}*\left(Q_{\arrmu^{\tr}}\right)] \rangle^{\low} \ne 0$$
and also the following corollary.
 \begin{cor}\label{secondinductivestep}
Given $C$ and $[L_C]$ as above, Proposition \ref{skeinmainprop} implies Proposition \ref{v=0skeinprop}.
\end{cor}
Also, the validity of Proposition \ref{v=0skeinprop} is independent of choices required in constructing the skein presentation $[L_C]$.
 
\subsection{Proof of Theorems \ref{mainthm1} and \ref{mainthm2}}

In this section, we prove our main results.
 
For Theorem \ref{mainthm2}, given a planar curve singularity $(C,p)$, let $M(C,p)$ denote the minimal number of blowups required for the total
transform of $C$ to have normal crossings singularities.  Proposition \ref{inductiveprop} and Corollary \ref{secondinductivestep} allow us to induct on $M(C,p)$, reducing to the base case of a nodal singularity; a further blowup allows us to reduce to the $v=0$ specialization for a node, which is Lemma \ref{basecase}.
 
To deduce Theorem \ref{mainthm1} from Theorem \ref{mainthm2}, we argue as follows.  Recall from Section \ref{coloredhomfly} that there is a shift in convention between the standard HOMFLY polynomial $\sP(L; v,s)$ and our colored HOMFLY polynomial $\sW(L, (1), \dots, (1); v,s)$.  So we know the statement after some sign and monomial shift, which we need to match with the predicted values.

By Proposition $6$ of \cite{oblomkov-shende}, we know the statement after the specialization
 $v= -1$.  This determines the exponents $\epsilon(C)$ and $a(C)$.  It remains to specify the monomial shift for $v$.
Let $\mu_{i}$ be the Milnor number of each branch of $C$ at $p$.
We have the classical formula
$$\mu - 1 = \sum_{i=1}^{r} (\mu_{i}-1) + 2 \mathrm{lk}(\cL_C).$$
It suffices to show that the monomial shift $v^{b(C)}$ for $\sW(L, (1), \dots, (1);v,s)$ is
$- \sum(\mu_{i}-1).$

It was explained how to calculate $b(C)$ at the end of Section \ref{algebraiclowdegree}.
First, it is additive over the branches of components of $\cL_C$, so it suffices to assume $C$ is irreducible.  
Its value on the Puiseux series $y(x)$ with Newton pairs $\{(p_0, q_0), \dots, (p_s,q_s)\}$
is given by
$$\sum_{i=0}^{s} - p_{i}q_{i} (\prod_{j> i} p_{j})^{2} + \sum_{i=0}^{s}q_{i}(\prod_{j > i} p_{j}) + \prod_{i=0}^{s} p_{i}$$
where the first summand comes from the writhe of $L_C$ in the definition of $\sW$ in Section \ref{coloredhomfly},
the second summand comes from the exponent at the end of Section \ref{algebraiclowdegree}, and the third summand comes 
from the calculation of $\langle Q_{\gamma}\rangle$.
By a classical calculation (e.g. Remark $10.10$ in \cite{milnor} after matching variables), this is precisely 
$1 - \mu(C).$

\subsection{Severi strata}
We now record an easy application of Theorem \ref{mainthm1}.  As mentioned in the introduction, given an integral curve $C$ with locally planar singularities, we can deduce that the topological Euler characteristic of its Hilbert scheme of points,
$C^{[n]}$, only depends on the singularity type of $C$.  

A nice consequence of this fact was pointed out to us by Vivek Shende.  
In what follows, let $$\calC \rightarrow B$$ 
denote a versal family of integral, locally planar curves of genus $g$, and let $\overline{B_h}$ denote the closure of the locus of curves with geometric genus $h \leq g$.
By Theorem A of \cite{vivek-severi}, the multiplicities of $\overline{B_h}$ at points $b \in B$ are determined by $\chi_{\tp}(\calC_b^{[n]})$.  This yields the following corollary.

\begin{cor}\label{severiobservation}
The multiplicity $m_h(b)$ of $\overline{B}_h$ at a point $b \in B$ is constant as $b$ varies along an equisingular locus of $B$.
\end{cor}

\section{Appendix: Hecke character calculation}

In this appendix, we prove Lemma \ref{spliceformula}.  The proof is completely parallel to the case of torus links, proven in 
\cite{morton-manchon} (using \cite{rosso-jones}) and \cite{lin-zheng}.  We will follow the proof in the latter paper closely, and refer the reader there for more details.

We first explain how to match conventions between \cite{morton-manchon} and \cite{lin-zheng} using \cite{lukac-thesis} as a reference.
Recall the Hecke algebra $H_n(z)$ of type $A_{n-1}$ is the algebra over $\QQ[z, z^{-1}]$ with generators $g_1,\dots, g_{n-1}$ and relations
\begin{enumerate}
\item[(i)] $g_i g_j = g_j g_i$ if $|i-j| \geq 2$
\item[(ii)] $g_i g_j g_i = g_j g_i g_j$ if $|i-j|= 1$
\item[(iii)] $(g_i - z)(g_i + z^{-1}) = 0$, for all $i$.
\end{enumerate}
After inverting $\frac{z^{k}-1}{z-1}$ for $2 \leq k \leq n$,
there is a decomposition
$H_n(z) = \prod_{\lambda \sideperp n} H_{\lambda}(z)$
where each $H_{\lambda}(z)$ is a matrix algebra.  Let 
$\pi_\lambda$ denote the corresponding central idempotent of $H_n(z)$ and $V_{\lambda}$ denote the corresponding simple module.

Recall also that, after identifying $z = s$ and enlarging our ring of scalars to $\Lambda$, we have an isomorphism
$\calH_n = H_n(z) \otimes \Lambda$
that identifies positively oriented braids on each side.  Explicit representatives in $\calH_n$ for primitive and central idempotents of $H_n(z)$ are constructed in \cite{lukac-thesis}.  It is shown there (Lemma $2.5.2$) that the primitive idempotents in $H_{\lambda}(z)$ have closure
$Q_{\lambda}$.

Given $X \in \calH_n$ and a partition $\rho \sideperp n$, the coefficient of $Q_{\rho}$ in the closure $\widehat{X} \in \calC_n$
is equal to the trace of $X$ in the representation $V_{\rho}$:
$$\Tr_{V_{\rho}}X = \Tr_{V_{\rho}} \pi_{\rho} X.$$
Indeed, both are trace functions on $H_{\rho}(z)\otimes \Lambda$ that agree on $\pi_{\rho}$.

Given relatively prime integers $(m,n)$, consider the tangle
$\sigma_{m}^{n}$ and partitions $\lambda \sideperp r$ and $\mu \sideperp s$.
Let $\widetilde{\sigma_{m}}^{n} \in H_{rm + s}(z)$ be the result
of cabling the first $m$ strands with $r$ strands and the last strand with $s$ strands.
Let $e_{\lambda}\in H_{rm}(z)$ and $e_{\mu}\in H_s(z)$ denote corresponding primitive idempotents and 
let $p_{\lambda,\mu} = e_{\lambda}^{\otimes r}\otimes e_{\mu}$ denote the corresponding projector
which by construction commutes with $\widetilde{\sigma_{m}}$.
We want to calculate the trace on $V_{\rho}$ of
$$X_{\lambda,\mu} = \pi_{\rho} \widetilde{\sigma_{m}}^{n} p_{\lambda,\mu}.$$
 
We now follow the argument from \cite{lin-zheng}.
If we take the tangle
$\sigma_{m}^{mn}$,
we see that the satellite
$$\widehat{\sigma_{m}^{mnK}}*(Q_{\lambda},\dots,Q_{\lambda},Q_{\mu}) = 
\widehat{\sigma_{1}^{nK}}*(Q_{\lambda}^{\cdot r}, Q_{\mu})
= v^{nK|\mu|} s^{-nK\kappa_{\mu}} T_{2}^{2nK}*(Q_{\lambda}^{\cdot r}, Q_{\mu}).$$
The last equality follows from the fact that 
$\beta_{2}^{2K}$ and $\sigma_{1}^{K}$ differ by adding $K$ full curls to the second strand
and that $Q_{\mu}$ is an eigenvector for the full curl with eigenvalue $v^{-|\mu|}s^{\kappa_{\mu}}$ by \cite{aiston-morton}.

We already know the expansion of $T_{2}^{2nK}(Q_1, Q_2)$ by Lemma \ref{torusformula}.
Therefore, for all $K>0$, we have
$$\Tr_{V_{\rho}} X_{\lambda,\mu}^{mK} = c_{\rho}v^{nK(|\mu| -|\rho|)} s^{nK(\kappa_{\rho} - \kappa_{\mu})}$$
where $Q_{\lambda}^{\cdot r}\cdot Q_{\mu} = \sum c_{\rho} Q_{\rho}.$
This implies that the eigenvalues of $X_{\lambda,\mu}^{m}$ are either
$v^{n(|\mu| -|\rho|)} s^{n(\kappa_{\rho} - \kappa_{\mu})}$ or $0$.
By the argument on page 11 of \cite{lin-zheng}, we have that
$$\Tr X_{\lambda,\mu} = a_{\rho} v^{n/m(|\mu| -|\rho|)} s^{n/m(\kappa_{\rho} - \kappa_{\mu})}$$
for some rational number $a_{\rho} \in \QQ$.

To determine $a_\rho$, we take the limit $v, s \rightarrow 1$.  This limit is well-defined since the idempotents of $H_n(z)$ extend over $z = 1$.
Under this limit, $\widetilde{\sigma_{m}}^{n}$  degenerates to the permutation 
$$\tau \times 1 \in \Sigma_{rm} \times \Sigma_s \subset \Sigma_{rm+s}$$
where $\tau$ acts cyclically on the $r$ $m$-tuples (since $(m,n)=1$) and trivially on the last $s$ terms.

The coefficient $a_{\rho}$ is the symmetric group character
$$\chi^{\rho}((\tau \times 1) e_{\lambda}^{\otimes m}\otimes e_{\mu}).$$
By restriction to the group algebra of $\Sigma_{rm}\times \Sigma_{s}$, this equals
$$\sum_{\sigma_{1}, \sigma_{2}} \mathrm{LR}^{\rho}_{\sigma_{1},\sigma_{2}} \chi^{\sigma_{1}}(\tau e_{\lambda}^{\otimes m})
\chi^{\sigma_{2}}(1\cdot e_{\mu})$$
where $\mathrm{LR}$ denotes Littlewood-Richardson coefficients.
The first term in the product is precisely the trace associated to the braid $\beta_{m}^{n}$ as calculated in \cite{lin-zheng}
and is given by the coefficient of $s_{\sigma_{1}}(z)$ in the plethysm $s_{\lambda}[p_m]$.  The second term is just
a delta function $\delta_{\sigma_{2},\mu}$.
Therefore,
$a_{\rho}$ is the coefficient of $s_{\rho}(z)$ in the symmetric function
$s_{\lambda}[p_{m}]\cdot s_{\mu}(z)$ which completes the proof.

Notice that Lemma \ref{spliceformula} implies integrality of $\frac{1}{m} (\kappa_{\rho} - \kappa_\mu)$ for all $\rho$ such that $a_\rho$ is nonzero.  This
can be proven directly using a skew version of the Murnaghan-Nakayama rule.

\end{document}